\documentclass{amsart}
\usepackage{amssymb}

\def\squareforqed{\hbox{\rlap{$\sqcap$}$\sqcup$}}
\def\qed{\ifmmode\squareforqed\else{\unskip\nobreak\hfil
\penalty50\hskip1em\null\nobreak\hfil\squareforqed
\parfillskip=0pt\finalhyphendemerits=0\endgraf}\fi\medskip}

\newcommand{\udot}{{}^{\textstyle .}}

\newtheorem{theorem}{Theorem}
\newtheorem{remark}{Remark}
\newtheorem{lemma}{Lemma}
\title{Maximal subgroups of ${}^2E_6(2)$ and its automorphism groups}
\date{Started on 13th March 2016.  This version 25th January 2018}
\author{Robert A. Wilson}

\address{School of Mathematical Sciences\\
Queen Mary University of London\\
London E1 4NS\\U.K.}
\email{R.A.Wilson@qmul.ac.uk}
\begin{document}

\begin{abstract}
We give a new computer-assisted proof of the classification
of maximal subgroups of the simple group ${}^2E_6(2)$ and its extensions
by any subgroup of the outer automorphism group $S_3$.
This is not a new result, but no earlier proof exists in the literature.
A large part of the proof consists of a computational analysis
of subgroups generated by an element of order $2$ and an element of order $3$.
This method can be effectively automated, and via statistical analysis
also provides a sanity check on results that may have been obtained
by delicate theoretical arguments.
\end{abstract}

\maketitle
\tableofcontents
\section{Introduction}
The maximal subgroups of ${}^2E_6(2)$ and its automorphism groups are part of the
folklore, but, as far as I am aware, 
no proof has ever been published. I believe the original result was
obtained by some subset of Peter Kleidman, Simon Norton 
and myself, some time around 1989,
but I cannot be entirely certain of that. In any case,
it seems to be worthwhile to provide a new proof, in the
interests of increasing confidence in the result.

Throughout, let $G$ be the simple group ${}^2E_6(2)$ of order
$2^{36}.3^9.5^2.7^2.11.13.17.19.$
Its automorphism group has shape $G.S_3$, in which the diagonal automorphism
group has order $3$, and the field automorphism has order $2$.
We aim to prove that the maximal subgroups of $G$, $G.2$, $G.3$ and $G.S_3$
are as listed in the Atlas \cite{Atlas}, apart from one or two minor errors in the 
structure of certain subgroups.

In Section~\ref{exist} we prove the existence of the
subgroups listed in the Atlas, with the required corrections,
and deduce the existence of $39$ isomorphism types of proper
non-abelian simple subgroups of $G$.
In Sections~\ref{cenout}, \ref{pne3} and \ref{3local} we classify the
$p$-local subgroups in $G.S_3$ for all relevant primes $p$.
In Section~\ref{ccexist}
we list as many conjugacy classes of simple
subgroups as we can, with justification given in
Sections~\ref{cent7}, \ref{cent5}, \ref{ccexistC3}, \ref{ccexistC2} and \ref{F22sub}. 
This includes a complete classification of
simple subgroups centralized by elements of order $5$ or $7$,
and hence yields  complete classification of non-abelian
characteristically simple subgroups that are not simple.

We begin the non-local analysis in Section~\ref{isolist} by
determining the proper non-abelian simple subgroups up to isomorphism.
We then embark on the main part of the classification, first using
structure constant analysis. Subgroups generated by 
$(2,3,n)$ triples 
for $n=5,7,11,13,17,19$ are classified in
Sections~\ref{A5}, \ref{Hurwitz}, \ref{11triples}, \ref{13triples}, \ref{17triples}, \ref{19triples}
respectively.
These results, summarized in Section~\ref{status}, 
give sufficient information in $23$ of the $39$ cases.
We then move on in Section~\ref{Monster}
to methods using the embedding of $2^2\udot {}^2E_6(2){:}S_3$
in the Monster. Essentially we use Norton's extensive work on
subgroups of the Monster to restrict the possibilities for maximal
subgroups of $G$. This deals with a further $12$ cases.
The final four cases are dealt with in Section~\ref{Baby},
and use detailed knowledge of subgroups of the Baby Monster,
much of it obtained by computational means.
The final Section~\ref{further} includes alternative, computer-free, proofs
for some of the results.

Our notation throughout follows the Atlas \cite{Atlas}.
In particular, $O$ is used for the generically simple groups of orthogonal
type, for example $O_5(3)$ denotes the simple group of order $25920$.

\section{Existence of the known maximal subgroups}
\label{exist}
There are
four conjugacy classes of maximal parabolic subgroups of $G$, all of which
extend to $G.S_3$, as follows:
\begin{itemize}
\item the centralizer $2^{1+20}{:}U_6(2)$ of a $2A$-involution (a $\{3,4\}$-transposition),
extending to $2^{1+20}{:}U_6(2){:}S_3$in $G.S_3$;
\item a four-group normalizer $2^{2+9+18}{:}(L_3(4)\times S_3)$,
extending to a group of shape $2^{2+9+18}{:}(L_3(4){:}S_3\times S_3)$ in $G.S_3$;
\item a $2^3$-normalizer $2^{3+4+12+12}{:}(A_5\times L_3(2))$: note that the structure
of this group is given incorrectly in the Atlas; it extends to
a group of shape $2^{3+4+12+12}{:}(A_5\times L_3(2)\times S_3)$ in $G.S_3$;
\item a group  of shape $2^{8+16}{:}O_8^-(2)$, extending to
$2^{8+16}{:}(O_8^-(2)\times3){:}2$ in $G.S_3$.
\end{itemize}

The following are maximal rank subgroups of $G$ which
can be read off from the Dynkin diagram, and also extend 
to $G.S_3$.
\begin{itemize}
\item $S_3\times U_6(2)$, of type $A_1+{}^2A_5$, extending 
to $S_3\times U_6(2){:}S_3$ in $G.S_3$;
\item $O_{10}^-(2)$, of type ${}^2D_5$, extending to $(O_{10}^-(2)\times 3){:}2$ in $G.S_3$;
\item $L_3(2)\times L_3(4)$, of type $A_1+A_1(q^2)$: 
this acquires an extra automorphism, giving
$(L_3(2)\times L_3(4)){:}2_1$ in $G$, and extending to $(L_3(2)\times L_3(4){:}S_3){:}2$ in $G.S_3$;
by looking at centralizers of outer element of order $14$ in $G.2$, we can see that the $L_3(4){:}S_3$
that centralizes $L_3(2)$ contains automorphisms of type $2_2$, in Atlas notation.
\end{itemize}

There are also the following subgroups of $G$ with their Lie type names:
\begin{itemize}
\item $F_4(2)$, in three conjugacy classes in $G$, extending to $F_4(2)\times 2$ in $G.S_3$;
\item $(3\times O_8^+(2){:}3){:}2$, of type $T_1+D_4$, extending to
$(3^2{:}2\times O_8^+(2)){:}S_3$ in $G.S_3$;
\item ${}^3D_4(2){:}3$, extending to ${}^3D_4(2){:}3\times S_3$ in $G.S_3$;
\item $U_3(8){:}3$, of type ${}^2A_1(q^3)$, extending to $(3\times U_3(8){:}3){:}2$ in $G.S_3$;
\item $3^2{:}Q_8\times U_3(3){:}2$, of type ${}^2A_1+G_2$, extending to $3^2{:}2S_4\times U_3(3){:}2$
in $G.S_3$;
\item $3^{1+6}{:}2^{3+6}{:}3^2{:}2$, of type $3({}^2A_1)$, extending to
$3^{1+6}{:}2^{3+6}{:}3^{1+2}{:}2^2$ in $G.S_3$.
\item $3^5{:}O_5(3){:}2$,  the normalizer of a maximal torus, extending to
$3^6{:}(2\times O_5(3){:}2)$ in $G.S_3$;
\end{itemize}
The normalizers of the
groups ${}^3D_4(2){:}3$
and $3^5{:}O_5(3){:}2$ are not maximal in $G$ or $G.2$,
as they are contained in $F_4(2)$ and $O_7(3)$ respectively.
As we shall see later on, their normalizers are however maximal when the diagonal
automorphism of order $3$ is adjoined.

Finally, as shown by Fischer, $G$ contains
\begin{itemize}
\item $Fi_{22}$, in three conjugacy classes in $G$, extending to $Fi_{22}{:}2$ in $G.S_3$;
\item $O_7(3)$, in three conjugacy classes in $G$, extending to $O_7(3){:}2$ in $G.S_3$.
\end{itemize}
It turns out that the normalizer of $O_7(3)$ is maximal only in $G.2$.

In particular, $G$ contains subgroups isomorphic to $U_6(2)$, $O_{10}^-(2)$,
$F_4(2)$, $U_3(8)$, and $Fi_{22}$.
Using knowledge of the maximal subgroups of these subgroups
\cite{BHRD,F42,KWF22}, we obtain the following 
list of $39$ isomorphism types of known simple subgroups of $G$:
\begin{itemize}
\item $A_5,A_6,A_7,A_8,A_9,A_{10},A_{11},A_{12},$
\item $L_2(7),L_2(8),L_2(11),L_2(13),L_2(16),L_2(17),L_2(25), L_3(3),L_3(4),L_4(3),$
\item $U_3(3),U_3(8),U_4(2),U_4(3),U_5(2),U_6(2),$
\item $O_7(3),O_8^+(2),O_8^-(2),O_{10}^-(2),S_4(4),S_6(2),S_8(2),$
\item ${}^2F_4(2)',{}^3D_4(2),G_2(3),F_4(2),M_{11},M_{12},M_{22},Fi_{22}.$
\end{itemize}
We shall show in Section~\ref{isolist}
below that every nonabelian simple proper subgroup of $G$ is
isomorphic to one of these $39$ groups.

\section{Centralizers of outer automorphisms}
\label{cenout}
The outer automorphism group of $G$ is $S_3$, and a number of maximal subgroups
may be obtained as centralizers of outer automorphisms of $G$, of order $2$ or $3$.
The outer automorphisms of order $2$ are given in the 
Atlas \cite{Atlas}, namely the elements in classes $2D$ and $2E$.
Those of order $3$ are not listed there, but are available in the GAP \cite{GAP}
character table of $G{:}3$. In Atlas notation these are elements in classes
$3D$, $3E$, $3F$, $3G$ and their inverses. Piecing together information from these
various sources, we find the structures of the centralizers in $G$ as follows;
\begin{itemize}
\item $C_G(2D)\cong F_4(2)$;
\item $C_G(2E)\cong [2^{15}]{:}S_6(2) < F_4(2)$;
\item $C_G(3D)\cong O_{10}^-(2)$;
\item $C_G(3E)\cong {}^3D_4(2){:}3$;
\item $C_G(3F)\cong U_5(2)\times S_3 < S_3 \times U_6(2)$;
\item $C_G(3G)\cong U_3(8){:}3$.
\end{itemize}

\section{$p$-local analysis for $p\ne3$}
\label{pne3}
In the simple group $G$, and in the group $G{:}3$ containing the diagonal automorphisms,
the maximal $2$-local subgroups are, by the Borel--Tits theorem, just the maximal parabolic subgroups,
which are well-known. 
Since all outer automorphisms of ${}^2E_6(2)$ are of diagonal or field type, all the
parabolic subgroups are normalized by the full outer automorphism group, and no more
maximal subgroups arise in any extension of $G$ as normalizers of $2$-subgroups of $G$.
In $G{:}2$ the $2$-local subgroups include the centralizers of
outer automorphisms of $G$ of order $2$,
which were considered in Section~\ref{cenout} above.

For the cyclic Sylow $p$-subgroups, that is, for $p=19,17,13$, or $11$, 
we have the following normalizers in $G$:
\begin{itemize}
\item $N(19)\cong 19{:}9<U_3(8){:}3$;
\item $N(17)\cong 17{:}8<O_{10}^-(2)$;
\item $N(13)\cong 13{:}12<{}^3D_4(2){:}3<F_4(2)$;
\item $N(11)\cong S_3\times 11{:}5<S_3\times U_6(2)$.
\end{itemize}
Extending to $G.S_3$ we obtain the following.
\begin{itemize}
\item $N(19)\cong (19{:}9\times 3){:}2 < (3\times U_3(8){:}3){:}2$;
\item $N(17)\cong 17{:}8\times S_3 < (3\times O_{10}^-(2)){:}2$;
\item $N(13)\cong 13{:}12\times S_3 < S_3\times {}^3D_4(2){:}3$;
\item $N(11)\cong S_3\times (3\times 11{:}5){:}2 < S_3\times U_6(2){:}S_3$.
\end{itemize}

We turn next to the Sylow subgroups of order $p^2$,
that is, $p=7$ or $5$. The relevant normalizers in $G$ as follows.
\begin{itemize}
\item $N(7A)\cong (7{:}3\times L_3(4)){:}2< (L_3(2)\times L_3(4)){:}2$.
\item $N(7B)\cong (7{:}3\times L_3(2)){:}2< (L_3(2)\times L_3(4)){:}2$.
\item the subgroup ${}^3D_4(2){:}3$ contains $7^2{:}(3\times 2A_4)$, which is the full
$7^2$-normalizer, since the $7^2$ is self-centralizing, and the stabilizer of any one of the
cyclic $7$-subgroups is $(7{:}3\times 7{:}3){:}2$, and there are two classes,
each of four such
cyclic subgroups.
\item $N(5)\cong (D_{10}\times A_8)\udot 2< O_{10}^-(2)$.
\item the centralizer of the Sylow $5$-subgroup is $5^2\times 3$, so the normalizer
lies in $(3 \times O_8^+(2){:}3){:}2$, and therefore has shape 
$(3 \times 5^2{:}4A_4){:}2$.
\end{itemize}
These normalizers extend to $G.S_3$ as follows.
\begin{itemize}
\item $N(7A)\cong (7{:}3\times L_3(4){:}S_3){:}2$;
\item $N(7B)\cong (7{:}3\times S_3 \times L_3(2)){:}2$;
\item $N(7^2)\cong S_3 \times 7^2{:}(3\times 2A_4)$;
\item $N(5)\cong (D_{10}\times (3\times A_8){:}2)\udot 2$;
\item $N(5^2)\cong (3^2{:}2 \times 5^2{:}4A_4){:}2$.
\end{itemize}

In particular, in this section we have proved the following.
\begin{theorem}
If $p\ge 5$, then 
no $p$-local subgroup is maximal in $G$ or any
extension of $G$ by outer automorphisms.
\end{theorem}
This leaves just the $3$-local subgroups to determine.

\section{The $3$-local subgroups}
\label{3local}
Since we have already dealt with the centralizers of outer automorphisms
of $G$ of order $3$, we may restrict attention to the normalizers of
elementary abelian $3$-groups inside $G$.
The three classes of subgroups of order $3$ in $G$ have the following normalizers:
\begin{itemize}
\item $N(3A)\cong S_3\times U_6(2)$;
\item $N(3B)\cong (3 \times O_8^+(2){:}3){:}2$;
\item $N(3C)\cong 3^{1+6}.2^{3+6}.3^2.2 = 3\udot(3^2{:}Q_8\times 3^2{:}Q_8\times 3^2{:}Q_8){:}3^2{:}2$.
\end{itemize}
Note that in this last case, the Atlas \cite[p. 191]{Atlas}
claims the quotient $3^2{:}2$ is isomorphic to
$3\times S_3$. That this is not the case is easily seen by comparing with the
subgroup $3^{1+6}{:}(2A_4\times A_4){:}2$ of $O_7(3)$.

Next we classify the elementary abelian subgroups of order $9$ in $G$.
Note first that there is a maximal torus with normalizer 
$3^5{:}O_5(3){:}2$, in which the isotropic points are of type $3C$ and the
$+$ points of type $3B$, and the $-$ points of type $3A$. Hence we get the fusion
of classes of elements of order $3$ in $3\times O_8^+(2)$ and $3\times U_6(2)$.
Alternatively, observe that the restriction of the $2$-modular Brauer character
of degree $78$ to the former has constituents of degrees $2+28+48$, from which
the fusion of outer elements of order $3$ can also be read off. 
The full fusion of elements of order $3$ is given in the following tables.
$$\begin{array}{l|ccc|cc|}
O_8^+(2): & 3ABC & 3D & 3E&3F&3G\cr
G: &3B & 3C & 3A&3B&3C\cr
\mbox{diagonal:} & 3AA & 3BB & 3CC &3BB &3CC
\end{array}$$
$$\begin{array}{l|ccc|}
U_6(2): & 3A & 3B & 3C\cr
G: & 3A & 3C & 3B\cr
\mbox{diagonal:} & 3BB & 3AA & 3CC
\end{array}$$
Thus there are $6$ classes of $3^2$ that contain $3A$ or $3B$ elements, and they have
one of the types $3AABB$, $3AAAC$, $3ABCC$, $3BBBC$, $3BBBB$, or $3BCCC$.
Note also that the centralizer in $O_8^+(2)$ of an element of class $3F$ is $U_3(3){:}2$,
while the centralizer of an element of class $3G$ is $3^2{:}2A_4$, in which the
normal $3^2$ consists of $3D$-elements, and all other $3$-elements are in $3E$.

In order to determine the class fusion from $N(3C)$ we need to
describe the structure of this group in some detail. There is a normal subgroup of index $18$
formed from the central product of three copies of $3^{1+2}{:}Q_8$.
Acting on this is (a) an element of order $3$ extending each copy to $3^{1+2}{:}2A_4$,
(b) and element of order $3$ permuting the three copies, and (c) an involution which
swaps two copies and extends the third to $3^{1+2}{:}2S_4$.
Now the outer automorphism of order $3$ of $G$ conjugates the elements of type (b) to
the products of (b) with (a). Hence we do not need to consider these other two cosets
separately.

As we know the $3^2$ groups of type $3AAAC$ and $3BBBC$, we can identify these with
elements in one copy of $3^{1+2}$, and elements diagonal between two copies,
respectively. Hence the elements diagonal between all three copies are in $3C$.
Now in $3^{1+2}{:}2A_4$, there are two types of outer elements: one is of order $3$,
and centralizes $3 \times 3^2{:}2$, while the other is of order $9$, and is self-centralizing.
It follows that
in the coset of type (a) we see four types of groups of order $3^2$, according to how
many of the factors $3^{1+2}$ contribute something of order $9$. If none,
we obtain a $3^2$ with centralizer $3^2 \times 3^3{:}2^3{:}3$. If one, we obtain an element
of order $9$ with centralizer of order $3^4.2^2$, and if two, an element of order $9$
and centralizer of order $3^4.2$. If all three, then we again obtain an elementary
abelian $3^2$, whose centralizer is elementary abelian of order $3^4$.
In the coset of type (b) we see two types of elementary abelian groups of order $3^2$,
one with centralizer $3^2\times 3^2{:}2A_4$, the other with centralizer
elementary abelian of order $3^4$.

The first $3^2$ of type (a) must be of type $3ABCC$, and the $3^5$ normal in its
centralizer is exactly the $3^5$ with normalizer $3^5{:}O_5(3){:}2$. The other
$3^2$ of type (a) is necessarily pure $3C$.
The first $3^2$ of type (b) is then forced to be of type $3BCCC$, and the second is again
pure $3C$. In particular, there are just two classes of $3^2$ of pure $3C$ type in $G.S_3$.
However, the second class splits into three classes in $G$.

We are now ready to classify the maximal $3$-local subgroups. The strategy is
to deal first with the elementary abelian groups of pure $3C$ type, then
those that contain $3A$ elements, and finally those that contain $3B$ elements
but not $3A$ elements.
\begin{lemma}
\label{3Clocal}
Every $3C$ pure elementary abelian group has normalizer in $G$ contained in one of the following
groups:
\begin{itemize}
\item $3^{1+6}{:}2^{3+6}{:}3^2{:}2$;
\item $3^5{:}O_5(3){:}2$;
\item $3^{3+3}{:}L_3(3)$ (three classes).
\end{itemize}
\end{lemma}
\begin{proof}
Suppose we have an elementary abelian $3$-group of pure $3C$ type containing
the normal $3$ in $N(3C)$. If it does not lie in $3^{1+6}$, then its centralizer
is elementary abelian of order $3^4$, and the centralizer contains $3B$ elements.
Moreover, the subgroup generated by these $3B$-elements has order $3^3$,
and contains a unique $3C$-pure $3^2$, which lies inside $3^{1+6}$.
Now some straightforward calculations shows that the latter group contains three classes 
of $3C$-pure $3^3$ under the action of the group $(Q_8\times Q_8\times Q_8){:}3^2{:}2$,
each with centralizer of order $3^6$. Such a centralizer
consists of $3^2\times 3^{1+2}$ inside $3^{1+6}$, together with an outer element
which is of type (b) in one case, and a product of type (a) and type (b) in the other two cases. Thus these three classes are fused in $G.3$.
Hence each such group has normalizer 
$3^{3+3}{:}L_3(3)$, which is already visible inside $O_7(3)$. Moreover,
any $3^2$ subgroup of this $3^3$ has centralizer of shape $(3\times 3^{1+4}).3^2$
which is the centralizer of an isotropic $2$-space in $3^5{:}O_5(3){:}2$.
Hence the normalizer of this $3^2$ is contained in the latter group,
and has shape $3^5{:}3^{1+2}{:}2S_4$.
\end{proof}

We turn next to the $3A$ elements.
\begin{lemma}
\label{3Alocal}
The normalizer of every elementary abelian subgroup of $G$ generated
by $3A$ elements lies in one of the following:
\begin{itemize}
\item the normalizer of a pure $3C$ type elementary abelian group;
\item $(S_3\times S_3\times U_4(2)){:}2$, contained in $O_{10}^-(2)$;
\item the group $3^5{:}O_5(3){:}2$.
\end{itemize}
\end{lemma}
\begin{proof}
From the above class fusion, 
we see that any elementary abelian $3$-group generated by $3A$ elements
either contains $3B$ elements, or has a unique subgroup of index $3$ that is
pure $3C$. The latter case corresponds to subgroups of $U_6(2)$ of pure $3B$ type,
and is covered by Lemma~\ref{3Clocal}. 
 In $U_6(2)$ the centralizer of a $3A$ element
is $3\times U_4(2)$, and the centralizer of a $3C$ element is a soluble group of order
$2^3.3^5$, and shape $3^4{:}(2\times A_4)$.

There is a unique $3^2$ of type $3AABB$, and it has normalizer
$(S_3\times S_3\times U_4(2)){:}2$, contained in $O_{10}^-(2)$. 
Any $3^3$ generated by $3A$ elements lies inside $3^2\times U_4(2)$,
corresponding to one of the three classes of subgroups of order $3$ in $U_4(2)$.
Hence there are three
types of $3^3$ generated by $3A$ elements and containing $3B$ elements.
One has a unique subgroup of order $3$ containing $3C$ elements, so we can ignore this case.
The other two have either 3 cyclic subgroups of type $3A$, and 6 of type $3B$,
or vice versa, and in both cases the centralizer contains a unique Sylow $3$ subgroup,
which is the torus of order $3^5$ described above.
\end{proof}

Hence we reduce to considering elementary abelian $3$-groups which contain 
$3B$ elements but no
$3A$ elements. 
\begin{lemma}
The normalizer in $G$ of every elementary abelian group that contains
$3B$ elements but no $3A$ elements lies in one of the following:
\begin{itemize}
\item $N(3B)$;
\item $3^2{:}Q_8 \times U_3(3){:}2$;
\item $N(3A)$;
\item $N(3C)$.
\end{itemize}
\end{lemma}
\begin{proof}
If the elementary abelian group lies in $3\times O_8^+(2)$, then
essentially the same argument as in Lemma~\ref{3Alocal}, with $3A$ and $3B$ interchanged, 
proves that we are in one of the cases already considered.
This is because in $O_8^+(2)$ we have $C(3A/B/C)\cong 3\times U_4(2)$, while
$C(3E)$ contains a unique elementary abelian $3^4$.
If it contains an outer element of $O_8^+(2){:}3$, then either this is in
class $3F$ or class $3G$ of $O_8^+(2){:}3$, in Atlas notation. In the $3F$ case,
we obtain a pure $3B$ type $3^2$, with normalizer
$3^2{:}Q_8\times U_3(3){:}2$. 
Now the class fusion from $U_3(3)$ to $O_8^+(2)$ goes via $S_6(2)$ classes $3B$ and
$3C$ respectively, so $O_8^+(3)$ classes $3D$ and $3E$. Thus any larger elementary
abelian $3$-group containing the $3^2$ of type $3BBBB$ either contains 
a unique cyclic subgroup containing $3A$ elements,
or contains a unique cyclic subgroup containing $3C$ elements.

In the $3G$ case, the $3^2$ has type $3BCCC$, so its normalizer lies in $N(3B)$.
Its centralizer in $O_8^+(2)$ is a group of shape $3^2{:}2A_4$, in which the normal
$3^2$ consists of $3D$ elements, and all other $3$-elements are in class $3E$.
Again, $3E$ elements fuse to $3A$ in $G$, so can be excluded. We can also assume that
our elementary abelian $3$-group contains no elements of class $3F$ in $O_8^+(2){:}3$.
But this implies that any remaining elementary abelian $3$-group contains $3B$ elements
but is not generated by them. Hence its normalizer is contained in a case already
considered.
\end{proof}

 This concludes the proof of the following theorem.

\begin{theorem}
Every $3$-local subgroup of $G$ is contained in one of the following subgroups:
\begin{itemize}
\item $N(3A)=S_3\times U_6(2)$;
\item $N(3B)=(3\times O_8^+(2){:}3){:}2$;
\item $N(3C)=3^{1+6}{:}Q_8^3{:}3^2{:}2$;
\item $N(3AABB)=(S_3\times S_3\times U_4(2)){:}2$, contained in $O_{10}^-(2)$;
\item $N(3B^2)=3^2{:}Q_8\times U_3(3){:}2$;
\item $3^5{:}O_5(3){:}2$, contained in $O_7(3)$;
\item $3^{3+3}{:}L_3(3)$, contained in $O_7(3)$ (three conjugacy classes).
\end{itemize}
\end{theorem}

\section{List of known conjugacy classes of simple subgroups}
\label{ccexist}
The main part of the proof below is a classification of simple subgroups
up to conjugacy. In order to facilitate this proof, we first list as
many conjugacy classes of simple subgroups as we can.
Tables~\ref{simpleslarge}, \ref{simples7} and \ref{simples5} 
contain one row for each conjugacy class in $G.S_3$,
of known simple subgroups $S$. We give the normalizer $N$ in $G.S_3$, and the number $n$
of conjugacy classes in $G$ (which may be $1,2,3$ or $6$), as well as a
maximal overgroup $M$ of $N$ in many cases.
The next several sections are devoted to proving the results
contained in these tables.

\begin{table}
\caption{\label{simpleslarge}Some simple subgroups, I}
(a) with elements of order $19$
$$\begin{array}{llll}
S&N&M&n\cr\hline
U_3(8) & (3\times U_3(8){:}3){:}2 && 1\cr\hline
\end{array}$$
(b) with elements of order $17$ but not $19$
$$\begin{array}{lllll}
S & N & M & n&\mbox{Notes}\cr
\hline
O_{10}^-(2) & (3\times O_{10}^-(2)){:}2 &&1&7A\cr
F_4(2) & F_4(2)\times 2 &&3\cr
S_8(2) & S_8(2)\times S_3 & (3\times O_{10}^-(2)){:}2 &1&7A\cr
S_8(2) & S_8(2)\times 2 & 2\times F_4(2) &3&7B\cr
O_8^-(2)& O_8^-(2){:}2 \times S_3 & (3\times O_{10}^-(2)){:}2 &1&7A\cr
O_8^-(2) & O_8^-(2){:}2 \times 2 & 2\times F_4(2) &3&7B\cr
S_4(4) & S_4(4){:}2\times S_3 & (3\times O_{10}^-(2)){:}2 & 1\cr
L_2(17) & L_2(17)\times S_3 & (3\times O_{10}^-(2)){:}2 & 1\cr
L_2(17) & L_2(17)\times 2 & 2\times F_4(2) & 3\cr
L_2(16) & L_2(16){:}2\times S_3 & (3\times O_{10}^-(2)){:}2 & 1\cr
L_2(16) & L_2(16){:}2\times S_3 & (3\times O_{10}^-(2)){:}2 & 1\cr
\hline
\end{array}$$
(c) with elements of order $13$ but not $17$ or $19$
$$\begin{array}{lllll}
S&N&M&n&\mbox{Notes}\cr\hline
{}^3D_4(2) & {}^3D_4(2){:}3\times S_3 &&1& 3A,3C,7B,7A\cr
{}^3D_4(2) & {}^3D_4(2){:}3\times 2 & 2\times F_4(2) &3& 3B,3C,7A,7B\cr
{}^2F_4(2)' & 2 \times {}^2F_4(2) & 2\times F_4(2) & 3\cr
L_2(25) & 2\times L_2(25)\udot 2 & 2\times F_4(2) & 3\cr
L_4(3) & 2\times L_4(3){:}2 & 2\times F_4(2) & 3\cr
L_3(3) & 2\times L_3(3){:}2 & 2\times F_4(2) & 3\cr
Fi_{22} & Fi_{22}{:}2 && 3&7B\cr
O_7(3) & O_7(3){:}2 && 3&7B\cr
G_2(3) & G_2(3){:}2 & Fi_{22}{:}2 & 3&7B\cr
L_2(13) & L_2(13){:}2 & Fi_{22}{:}2 & 3&7B\cr
\hline
\end{array}$$
(d) with elements of order $11$ but not $13$, $17$ or $19$
$$\begin{array}{lllll}
S&N&M&n&\mbox{Notes}\cr\hline
A_{12} & (3\times A_{12}){:}2 & (3\times O_{10}^-(2)){:}2 & 1&7A\cr
A_{11} & (3\times A_{11}){:}2 & (3\times O_{10}^-(2)){:}2 & 1&7A\cr
M_{12} & M_{12}\times 3 & (3\times O_{10}^-(2)){:}2 & 2\cr
M_{11} & M_{11}\times 3 & (3\times O_{10}^-(2)){:}2 & 2\cr
M_{11} & M_{11}\times 3 & (3\times O_{10}^-(2)){:}2 & 2\cr
L_2(11) & S_3 \times (3\times L_2(11)){:}2 & U_6(2){:}3\times S_3 & 1 & 2B,3B\cr
L_2(11) & (L_2(11)\times 3){:}2 & (3\times O_{10}^-(2)){:}2 & 1 & 2C,3C\cr
M_{22} & S_3 \times M_{22}{:}2 & S_3\times U_6(2){:}S_3 & 3&7B\cr
U_6(2) & S_3 \times U_6(2){:}S_3 &&1&7B\cr
U_5(2) & S_3 \times (3\times U_5(2)){:}2 & S_3 \times U_6(2){:}S_3 &1\cr
\hline
\end{array}$$
Note: in fact this is a complete list of conjugacy classes of the given
simple groups. This fact is proved in this paper.
\end{table}

\begin{table}
\caption{\label{simples7}Some simple subgroups, II}
(e) with elements of order $7$ but not $11,13,17,19$
$$\begin{array}{lllll}
S&N&M&n&\mbox{Notes}\cr\hline
O_8^+(2) & O_8^+(2){:}3^{1+2}{:}2^2 &&1&7A\cr
O_8^+(2) & 2\times O_8^+(2){:}S_3 & 2\times F_4(2) &3&7B\cr
A_{10} & S_3\times S_{10} & (3\times O_{10}^-(2)){:}2 & 1&7A\cr
A_{10} & 2\times S_{10} & 2\times F_4(2) & 3&7B\cr
A_9 & (A_9\times 3^2{:}2){:}2 & O_8^+(2){:}3^{1+2}{:}2^2 & 1&7A\cr
A_9 & S_9\times 2 & 2\times F_4(2) & 3&7B\cr
A_8 & (A_8\times A_5\times 3){:}2^2 & (3\times O_{10}^-(2)){:}2 & 1&7A\cr
A_8 & (A_8\times A_4\times 3){:}2^2 & (3\times O_{10}^-(2)){:}2 & 1&7A\cr
A_7 & (A_7\times A_5\times 3){:}2^2 & (3\times O_{10}^-(2)){:}2 & 1&7A\cr
A_7 & S_7\times S_3 & N(3A) & 3 & 3A,3C,7B\cr
A_7 & S_7\times S_3 & N(3A) & 3 & 3C,3C,7B\cr
S_6(2) & S_3\times S_3\times S_6(2) & (3\times O_{10}^-(2)){:}2 & 1&7A\cr
S_6(2) & 2\times S_3\times S_6(2) & 2\times F_4(2) & 3 &7B\cr
S_6(2) & 2\times S_6(2) &2^{1+20}{:}U_6(2){:}S_3 & 6 & 7B\cr
L_2(8) & S_3\times S_3\times L_2(8){:}3 & S_3\times {}^3D_4(2){:}3 & 1 & 7A\cr
L_2(8) & 2 \times L_2(8){:}3& 2 \times F_4(2) & 3 & 7A\cr
L_2(8) & S_3\times S_3\times L_2(8){:}3 & S_3\times U_6(2){:}S_3 & 1&7B\cr
L_2(8) & 2 \times L_2(8){:}3 & N(2A) & 3 & 7B\cr
L_2(8) & 2 \times L_2(8){:}3 & N(2A) & 2 & 7B\cr
L_2(8) & 2^2\times L_2(8) & N(2A) & 3 & 7B\cr
L_2(8) & 2^2\times L_2(8) & N(2A) & 3 & 7B\cr
L_2(8) & 2 \times L_2(8) & N(2A) & 6 & 7B\cr
L_2(8) & 2 \times L_2(8) & N(2A) & 6 & 7B\cr
U_4(3) & S_3 \times U_4(3).2^2 & S_3\times U_6(2){:}S_3 & 3 & 7B\cr
L_3(4) & (L_3(2)\times L_3(4){:}2){:}S_3 &&1&2B,3B,7B\cr
L_3(2) & L_3(2){:}2 \times L_3(2){:}2 & (L_3(2)\times L_3(4){:}S_3){:}2 & 3 & 2B,3B,7B\cr
L_3(2) & (L_3(2)\times L_3(4){:}2){:}S_3 &&1&2A,3A,7A\cr
L_3(2) & L_3(2) \times 2^4{:}(3\times A_5){:}2 & N(2A^3) & 1 & 2B,3A,7A\cr
L_3(2) & L_3(2){:}2\times S_3 &(3\times O_{10}^-(2)){:}2&1& 2C,3C,7A\cr
L_3(2) &L_3(2){:}2&(L_3(2)\times L_3(4){:}S_3){:}2&3& 2C,3C,7B\cr
U_3(3) & U_3(3){:}2\times 3^2{:}2S_4 && 1&7A\cr
U_3(3) & U_3(3){:}2 \times 2\times S_3 & 2\times F_4(2) & 3 & 7B\cr
\hline
\end{array}$$
Note: This list is claimed to be complete in the cases
$O_8^+(2)$, $A_{10}$, $A_9$, $A_8$, $A_7$, $S_6(2)$, $L_2(8)$, but not
necessarily in the cases $L_2(7)$,
$L_3(4)$, $U_3(3)$, and $U_4(3)$.
\end{table}
\begin{table}
\caption{\label{simples5}Some simple subgroups, III}
(f) without elements of order $7,11,13,17,19$
$$\begin{array}{lllll}
S&N&M&n&\mbox{Notes}\cr\hline
U_4(2) & 3^2{:}D_8 \times (3\times U_4(2)){:}2 & (3\times O_{10}^-(2)){:}2 & 1& U_4(2)<U_6(2)\cr
U_4(2) & 2\times S_3\times U_4(2){:}2 & S_3\times U_6(2){:}S_3 & 3 & O_6^-(2)<U_6(2)\cr
A_6 & S_6\times S_6\times S_3 & (3\times O_{10}^-(2)){:}2 & 1\cr
A_6 & (A_6.2\times L_3(2)){:}2 & (L_3(4){:}S_3\times L_3(2)){:}2 & 3& 2B,3B\cr
A_5 & S_5\times S_6 \times S_3 & (3\times O_{10}^-(2)){:}2 & 1& 2B,3A\cr
A_5 & (A_5\times 3\times A_8){:}2^2 & (3\times O_{10}^-(2)){:}2 & 1 & 2B,3B\cr
A_5 & S_5\times 2^3{:}L_3(2) & N(2A^3) & 3 & 2B,3B\cr
A_5 & S_5\times S_3 & (3\times O_{10}^-(2)){:}2& 1 & 2C,3C\cr
A_5 & (A_5\times 2^3{:}S_4){:}2 & & 3 & 2C,3B\cr
A_5 & (A_5 \times 2^4{:}3^2{:}2){:}2 & & 1&2C,3A\cr
A_5 & (A_5\times [2^7.3]).2& & 3 & 2B,3B\cr
A_5 & (A_5\times [2^5.3^2]).2 & & 1 & 2B,3B\cr
A_5 & A_5 \times [2^6.3] & & 3 & 2C,3B\cr
A_5 & (A_5\times [2^4]).2 & & 3 & 2C,3B\cr
\hline
\end{array}$$
Note: This list is claimed to be complete in the cases $A_5$ and $U_4(2)$, but not
necessarily in the case $A_6$.
\end{table}
In fact, most of the known simple subgroups lie in the centralizer
of some (inner or outer) automorphism, so we deal with these first.
We start by centralizing elements of order $7$, followed by $5$,
$3$ and $2$. We conclude with some subgroups of $Fi_{22}$.
Structures of normalizers are given in $G.S_3$ unless otherwise stated.

\section{Simple subgroups centralizing an element of order $7$}
\label{cent7}
The centralizers of elements of order $7$ in $G$ are $C(7A)\cong 7\times L_3(4)$ and
$C(7B)\cong 7\times L_3(2)$. Power maps show that 
a $7A$ element commutes with elements in classes
$2B,3B,4D,4E,4F,5A$ and that a $7B$ element commutes with elements
of classes $2A$ and $3A$, and $4A$. 
It is clear from the $7$-local analysis that in $L_3(2)\times L_3(4)$
the $L_3(2)$ factor contains $7A$ elements and the $L_3(4)$ contains $7B$ elements.
It follows immediately that every $L_3(2)\times L_3(2)$ contains one factor with $7A$-elements
and the other with $7B$-elements, so there is no automorphism swapping the two factors.

We have $N(7B)=(7{:}3\times S_3\times L_3(2)){:}2)$ in $G.S_3$, so the only
simple group centralizing a $7B$ element is an $L_3(2)$ with normalizer
$(L_3(2)\times L_3(4){:}2){:}S_3$. Also $N(7A)=(7{:}3\times L_3(4){:}3{:}2){:}2$,
and the normalizers of the simple subgroups of $L_3(4){:}D_{12}$ are
\begin{itemize}
\item One class of $A_6\udot 2^2$.
\item One class of $L_3(2){:}2\times 2$.
\item One class of $S_3\times S_5$.
\item One class of $S_5$.
\end{itemize}
All these simple groups centralize $L_3(2)$ in $G$, but may centralize more.
First, the $L_3(2)$ contains a $7B$ element, so the centralizer does not grow,
and the normalizer in $G.S_3$ is $L_3(2){:}2\times L_3(2){:}2$. The other groups
all contain elements of order $5$, so the centralizer lies in $(3\times A_8){:}2$.
The $A_5$ in $S_3\times S_5$ must be the one with normalizer
$(A_5\times 3\times A_8){:}2^2$. The other two normalizers do not contain
the full $A_8$, but in the case of $A_6$ the normalizer contains $L_3(2)$ 
and an involution that normalizes the $A_6$ and centralizes the element
of order $5$, and hence the normalizer is $(L_3(2)\times A_6.2){:}2$.
In the case of the last $S_5$, the normalizer is in fact 
$2^3{:}L_3(2)\times S_5$.

\begin{remark}
From the class fusion to the factors of $L_3(2)\times L_3(4)$ we can compute the
restriction of the character of degree $1938$ to these factors. There is then only one
way to fit them together into characters of the direct product, thus:
$$1\otimes 1 + 3a\otimes 45a + 3b\otimes 45b + 6\otimes (1^2+20^2)
+7\otimes 35abc + 8\otimes(1+20+64).$$
Hence we can compute the character value on the diagonal involutions to be $18$,
and on the diagonal elements of order $3$ to be $-6$. It follows that these
diagonal elements are in classes $2C$ and $3C$.
In particular, the diagonal  copies of $L_3(2)$ in $L_3(2)\times L_3(4)$ are of type
$(2C,3C,7A)$ and $(2C,3C,7B)$. It is not immediately obvious what the centralizers
of these copies of $L_3(2)$ are, and we shall come back to this problem later.
\end{remark}

We conclude this section by summarizing the consequences for the classification of
maximal subgroups with non-simple minimal normal subgroups.
\begin{theorem}
There is exactly one class of characteristically simple subgroup
of $G.S_3$ that has order divisible by $7$ and is not simple.
Such groups are isomorphic to $L_3(2)\times L_3(2)$ and have
normalizers $L_3(2){:}2\times L_3(2){:}2$, all of which are contained in
$(L_3(2)\times L_3(4){:}S_3){:}2$. The class splits into three
classes in $G$.
\end{theorem}

\section{Simple subgroups centralizing an element of order $5$}
\label{cent5}
In $G.S_3$ we have $N(5A)=(D_{10}\times 3\times A_8){:}2^2$, and the normalizers of the
simple subgroups of $S_8$ are
\begin{itemize}
\item One class of $S_7$.
\item One class of $S_6\times 2$.
\item One class of $S_5\times 3$.
\item One class of $S_5\times 2$.
\item One class of $L_3(2){:}2$.
\item One class of $L_3(2)$.
\end{itemize}
All these simple groups centralize $3\times A_5$, but may centralize more.
The ones with order divisible by $7$ have centralizers inside $L_3(4){:}S_3$.
Clearly the $L_3(2)$ in $L_3(2){:}2$ is the one which centralizes $L_3(4)$.
The other $L_3(2)$ centralizes $2^4{:}A_5$, and has normalizer
$L_3(2)\times 2^4{:}(3\times A_5){:}2$, lying inside the maximal parabolic
subgroup of shape $2^3.2^4.2^{12}.2^{12}.(S_5\times L_3(2)\times S_3)$ in $G.S_3$,
and remaining a single class in $G$.

Since there is no parabolic subgroup containing $A_7\times A_5$, the
$A_7$ has normalizer just $(3\times A_5\times A_7){:}2^2$. The $A_5$
in $S_5\times S_3$ centralizes a $3B$ element and is therefore conjugate 
in $O_8^+(2){:}S_3$ to
the one that centralizes $A_8$.
The $S_6$ lies in $O_{10}^-(2)$ acting as $S_4(2)$, so centralizes $S_6$.
Hence the normalizer of the $A_6$ is $S_6\times S_6\times S_3$ in $G.S_3$.
The remaining $A_5$ also centralizes the same $S_6$, so has normalizer
$S_5\times S_6\times S_3$.

Power maps give the class fusion from $A_8$ to $G$. The elements of
$A_8$-classes $2A,2B,3A,3B$ fuse to $G$-classes $2A,2B,3B,3A$ respectively.
In particular, the $5$-point $A_5$ in $A_8$ is of type $(2B,3B)$, while the $6$-point
$A_5$ is of type $(2B,3A)$.

Since $S_6$ is maximal in $A_8$, it follows that the centralizer in both $O_{10}^-(2)$ and
$G$ of this $A_5$ of type $(2B,3A)$
is exactly $S_6$. It follows that there are just two 
classes of $A_5\times A_5$ in which the two factors are conjugate in $G$.
One of these has factors of type $(2B,3B)$ and centralizer of order $3$, so its
normalizer in $G$ is contained in $(3\times O_8^+(2){:}3){:}2$. The other has factors
of type $(2B,3A)$, and has trivial centralizer, and its normalizer 
in $G$ is contained in
$(S_6\times S_6){:}2$ in $O_{10}^-(2)$. Similarly, there is a unique conjugacy class of
$A_6\times A_6$, with normalizer $(S_6\times S_6){:}2$ in $G$, contained in $O_{10}^-(2)$. 

We summarize the consequences for maximal subgroups with non-simple 
minimal normal subgroups.
\begin{theorem}
There are four classes of characteristically simple subgroups of $G.S_3$
that have order divisible by $5$ and are not simple, as follows.
\begin{itemize}
\item One class of $A_6\times A_6$, with normalizer $S_6\wr 2 \times S_3$,
contained in $(3\times O_{10}^-(2)){:}2$.
\item One class of $A_5\times A_5$ with normalizer contained in
$N(3B)$.
\item Two classes of $A_5\times A_5$ with normalizer contained in
$S_6\wr 2 \times S_3$.
\end{itemize}
\end{theorem}

\section{Simple subgroups centralizing an element of order $3$}
\label{ccexistC3}
\subsection{Groups centralizing a $3G$ element}
\label{C3G}
We have $N(3G)=(3\times U_3(8){:}3){:}2$, and the only proper non-abelian
simple subgroup of $U_3(8)$ is a single class of $L_2(8)$.
Now a $7A$ element commutes with $L_3(4){:}S_3$, which has two classes
outer $3$-elements, fusing to $3D$ and $3E$ respectively.
Similarly, a $7B$ element commutes with $(L_2(7)\times3){:}2$
which also has two classes of outer $3$-elements, fusing to $3E$ and $3G$
respectively. In particular, the $7$-elements in $U_3(8)$ lie in $7B$,
so the $L_2(8)$ is of type $(2C,3C,7B)$.
Hence the $L_2(8)$ centralizes in $G.S_3$ a proper 
subgroup of $(L_2(7)\times 3){:}2$, containing $3^2{:}2$.
The normalizer therefore contains $L_2(8){:}3\times 3^2{:}2$, 
in which the centralizing $3^2$ contains one cyclic subgroup of 
type $3A$, one of type $3E$, and two of type $3G$.
Hence there is also an element of the normalizer interchanging
the two subgroups of type $3G$, so
the normalizer in $G.S_3$
contains $S_3\times S_3\times L_2(8){:}3$.
But $S_3\times S_3$ is maximal in $(L_2(7)\times 3){:}2$,
so this is the full normalizer.

\subsection{Groups centralizing a $3E$ element}
\label{C3E}
Next consider the case $3E$, where we have $N(3E)={}^3D_4(2){:}3\times S_3$.
Now ${}^3D_4(2)$ contains just three classes of proper non-abelian
simple subgroups, with normalizers $S_3\times L_2(8)$, $U_3(3){:}2$,
and $(7\times L_2(7)){:}2$. The $7$-elements in $L_2(7)$ and $U_3(3)$ are
in ${}^3D_4(2)$ class $7D$, so $G$-class $7A$, while those in $L_2(8)$
are in $7B$. (Warning: there are two classes of ${}^3D_4(2)$,
with opposite fusion of $7$-elements. The one under consideration here
contains $3A$-elements, so $7ABC$ fuses to $7B$ while $7D$ fuses to $7A$.
The other one contains $3B$-elements.
These facts are verified below using structure costants.)
The above group $L_2(8)$ therefore has type $(2C,3C,7B)$ and its normalizer
in $G.S_3$ is $S_3\times S_3\times L_2(8){:}3$. This is the
same as the one in $U_3(8)$.
The above group $L_2(7)$ has type $(2A,3A,7A)$ and centralizes $L_3(4)$.
Its full normalizer in $G.S_3$ is $(L_3(2)\times L_3(4){:}S_3){:}2$.
The centralizer of the $U_3(3)$ is a proper subgroup of $L_3(4){:}S_3$,
and is therefore $3^2{:}2S_4$, since we know that there is a
subgroup $3^2{:}2S_4 \times U_3(3){:}2$ in $G.S_3$.

\subsection{Groups centralizing a $3F$ element}
\label{C3F}
Next consider the case $3F$, where we have $N(3F)=S_3\times(3\times U_5(2)){:}2$.
Now $U_5(2)$ contains one class each of $U_4(2)$, $L_2(11)$ and $A_6$,
and three classes of $A_5$. This $L_2(11)$ has normalizer
$S_3\times(3\times L_2(11)){:}2$ in $G.S_3$. The other groups could have
larger normalizers in $G.S_3$ than they do in $N(3F)$.
In any case, the centralizer in $G$ of any of these
groups is a subgroup of $A_8$, containing the given $S_3$.
Note that the normal $3^2$ in $N(3F)$ contains one cyclic subgroup of
type $3A$, one of type $3F$, and two of type $3D$.

In the case of $U_4(2)$, we already see inside the group $S_3\times U_6(2){:}S_3$
a subgroup $S_3\times S_3\times (3\times U_4(2)){:}2$, which contains a
subgroup $S_3\times S_3$ of $A_8$. This $S_3\times S_3$ contains two cyclic subgroups
of type $3A$ and two of type $3B$, and an extra automorphism of order $2$ is
realised inside $N(3D)$. The only proper subgroup of $A_8$ containing
$3^2{:}D_8$ is $S_6$, but $U_4(2)$ does not centralize an element of order $5$.
Hence the full $U_4(2)$ normalizer in $G.S_3$ is
$3^2{:}D_8\times (3\times U_4(2)){:}2$. 
The normal $3^2$ is of type $3AABB$, and normal $3$ is of type $3D$.
The diagonal elements are of type $3D$ and $3F$, corresponding to the
$3B$ and $3A$ elements repectively. The $13$ subgroups of order $3^2$ are
one of type $3AABB$, two each of types $3ADFF$ and $3BDDD$, and four
each of types $3ADDF$ and $3BDFF$.

The same argument applied to $A_6$ allows the possibility that 
the centralizer of the $A_6$ grows to $S_6$.
Indeed, we can see this centralizer inside $O_{10}^-(2)$, so we have
$S_6\times S_6\times S_3$ as the normalizer of the $A_6$ in $G.S_3$.
Moreover, there is also an automorphism swapping the two factors of
$S_6$, also realised inside $(3\times O_{10}^-(2)){:}2$.
Of the two types of $A_5$ in $A_6$, one centralizes the full $A_8$,
while the other centralizes only $S_6$. Hence we obtain normalizers
$(3\times A_5\times A_8){:}2^2$ and $S_5\times S_6\times S_3$ respectively.

\subsection{Groups centralizing a $3D$ or $3B$ element}
\label{C3D}
The $3D$-normalizer $(3\times O_{10}^-(2)){:}2$ contains
a subgroup of index $3$ of the $3B$-normalizer, so in particular
contains all non-abelian simple subgroups of $N(3B)$.

The group $O_{10}^-(2)$ contains a large number of classes of simple
subgroups. Consider first those with order divisible by $17$, that is
$S_8(2)$, $O_8^-(2)$, $S_4(4)$, $L_2(16)$ and $L_2(17)$. There are two
classes of $L_2(16)$ in $O_{10}^-(2)$, and one class of each of the
other groups. Since the Sylow $17$-normalizer lies inside $N(3D)$,
the normalizers are all easy to write down, and they are 
$S_8(2)\times S_3$, $O_8^-(2){:}2\times S_3$, $S_4(4){:}2\times S_3$,
$L_2(17)\times S_3$, and two classes of $L_2(16){:}2\times S_3$.

Next consider the simple subgroups with order divisible by $11$, that is
$A_{12}$, $A_{11}$, $M_{12}$, $M_{11}$, $U_5(2)$ and $L_2(11)$.
There is a single class each of $A_{12}$, $A_{11}$ and $U_5(2)$,
with normalizers in $N(3D)$ of shape $(3\times A_{12}){:}2$,
$(3\times A_{11}){:}2$ and $(3^2\times U_5(2)){:}2$. The last of these
has a larger normalizer in $G.S_3$, as already discussed in Subsection~\ref{C3F}.
Since $A_{11}$ does not lie in $U_6(2)$, it does not centralize an
involution, so the above groups are the full
normalizers of $A_{11}$ and $A_{12}$ in $G.S_3$.

Now $O_{10}^-(2)$ contains two classes of $M_{12}$, fused by the
outer automorphism. Hence there is one such class of $M_{12}$ in
$G.S_3$, each with normalizer $3\times M_{12}$, and splitting into
two classes in $G$. In the case of $M_{11}$, there are two classes
fixing one of the $12$ points on which $A_{12}$ acts, and two
acting transitively, each pair fused by the outer automorphism.
Hence there are two classes of $M_{11}$ in $G.S_3$, each with
normalizer $3\times M_{11}$, and splitting into four classes in $G$.
The subgroups $L_2(11)$ are again of two types: transitive and
intransitive on $12$ points. The former also lies in $U_5(2)$, so has
normalizer $S_3\times (3\times L_2(11)){:}2$ as discussed 
in Subsection~\ref{C3F} above.
The latter has normalizer $(3\times L_2(11)){:}2$ contained in
$N(3D)$. (We need to prove that it does not also centralize
a $2A$-element: a somewhat subtle question. In fact, there is a unique
class of $L_2(11)$ in $U_6(2)$, and such an $L_2(11)$ acts on the
$2^{20}$ factor of $2^{1+20}$ as two copies of an absolutely irreducible
$10$-dimensional module. The $1$-cohomology of this module is trivial,
so there is only one conjugacy class of $L_2(11)$ in $2^{1+20}{:}L_2(11)$.)

Next consider subgroups with order divisible by $7$ but not
by $11$ or $17$. This includes $A_{10}$, $A_9$, $A_8$, $A_7$,
$L_2(7)$, $L_2(8)$, $O_8^+(2)$, $S_6(2)$, $U_3(3)$. The $7$-elements
here are in class $7A$, since they centralize elements of order $5$.
There is a unique class of $O_8^+(2)$ in $O_{10}^-(2)$, and such a group
centralizes a unique $3B$ element in $G$. The normalizer is
therefore equal to $N(3B)$.
 Since this normalizer includes a triality automorphism of $O_8^+(2)$,
we can use this automorphism to prove conjugacy in $G.S_3$ of various
groups. In particular there is a unique class of $S_6(2)$ that
centralizes a $3D$ element.

Since the $7A$ element centralizes $L_3(4){:}S_3$, in which the
$3D$ element we are centralizing lies in class $3B$, it makes sense
to classify the subgroups of $L_3(4){:}S_3$ that contain a $3B$-element.
Now in $L_3(4){:}3$, the maximal subgroups containing an outer
element of order $3$, are $2^4{:}(3\times A_5)$ (two classes),
$7{:}3\times 3$ and $3^2{:}2A_4$. There are two classes of outer $3$-elements
in $3^2{:}2A_4$ (up to inversion), of which those that centralize involutions
are in $L_3(4)$-class $3B$. All outer elements of order $3$ in
$7{:}3\times 3$ lie in $L_3(4)$ class $3C$. Notice also that the
$7$-centralizer in $(3\times O_{10}^-(2)){:}2$ is $7 \times (3\times A_5){:}2$.

There is a unique class of $A_{10}$ in $N(3D)$, with normalizer $S_3\times A_{10}$.
Since the centralizer of a $5$-cycle therein is $5^2\times S_3$, this centralizer
cannot grow in $G.S_3$. There are two classes of $A_9$ in $N(3D)$, but these
are conjugate by a triality automorphism of $O_8^+(2)$. The normalizer is
$(3^2{:}2\times A_9){:}2$ contained in $(3^2{:}2\times O_8^+(2){:}3){:}2$. 
The precise structure of this normalizer can be seen from the subgroup
$(A_5\times A_{12}){:}2$ of the Monster, and is the unique subgroup of
index $2$ in $S_3\times S_3\times S_9$ that is not a direct product.

Similarly for the case of $S_6(2)$, there are two classes in $N(3D)$,
fused by triality. Since $S_6(2)$ contains elements of order $15$, the centralizer
in $G$ lies somewhere between $S_3$ and $A_5$. But it is not $A_5$, and $S_3$
is maximal in $A_5$, so it is $S_3$. Hence the normalizer in $G.S_3$ is
$S_3\times S_3\times S_6(2)$.

This leaves $A_7$, $A_8$, $L_2(7)$, $L_2(8)$, $U_3(3)$, which are significantly harder.
We shall not attempt a complete classification at this stage.
Note however that there is a maximal subgroup $L_2(7){:}2$ of
$O_8^-(2){:}2$. This $L_2(7)$ is of type $(2C,3C,7A)$ in $G$ and does not
centralize any involution in $G$. Hence its normalizer in $G.S_3$ is
exactly $L_2(7){:}2\times S_3$.

Finally, we consider the simple groups divisible by no prime greater than $5$,
that is $A_5$, $A_6$, and $U_4(2)$. There are two classes of $U_4(2)$: those that
act on the $10$-space as $O_6^-(2)$, and those that act as $U_4(2)$. But by triality
of $O_8^+(2)$, these are conjugate in $G.S_3$. Moreover, they are conjugate to
the subgroup $U_4(2)$ of $U_5(2)$ already considered above.

We leave $A_5$ and $A_6$ for the time being.

\subsection{Groups centralizing a $3A$ element}
\label{C3A}
The maximal subgroup $N(3A)=S_3\times U_6(2){:}S_3$ contains many interesting subgroups.
As it contains $7B$-elements, many of these are not 
visible elsewhere.
In particular $U_6(2)$ contains $U_4(3)$. The normalizer in $G.S_3$ is
$S_3 \times U_4(3).2^2$, and the class splits into three classes in $G$.
It also contains $M_{22}$, with normalizer $S_3\times M_{22}{:}2$,
and splitting into three classes in $G$. It also contains
$S_6(2)$, with normalizer $S_3\times 2\times S_6(2)$, splitting into
three classes in $G$, and similarly $U_3(3)$, with normalizer
$S_3\times 2 \times U_3(3){:}2$, also splitting into three classes in $G$.

In addition to these four groups, there are unique classes in $U_6(2){:}S_3$ of each of
$A_8$, $L_2(11)$, $U_5(2)$ and $L_2(8)$, all of which we have
already seen. There are two classes of $U_4(2)$, one acting as $U_4(2)$,
the other as $O_6^-(2)$. The former we have already seen;
the latter group centralizes only $S_3\times 2$ in $G.S_3$,
and the class splits into three classes in $G$.
There are two classes of $A_7$ in $U_6(2){:}S_3$, one inside 
$S_6(2)$, the other inside $U_4(3)$. 
In both cases the normalizer is $S_3\times S_7$. The first contains both
$3A$ and $3C$ elements, the other only $3C$ elements.
There are also subgroups $A_5$, $A_6$,
$L_3(2)$, $L_3(4)$, which we shall not completely classify here.

\section{Simple subgroups centralizing an involution}
\label{ccexistC2}
\subsection{Groups centralizing a $2D$ element}
\label{C2D}
The centralizer of a $2D$-element in $G.S_3$ is $2\times F_4(2)$.
Now $F_4(2)$ has an outer automorphism of order $2$ that is \emph{not}
realised in $G.S_3$. Hence there are pairs of automorphic subgroups of
$F_4(2)$ that behave completely differently in $G$.
This applies in particular to the subgroups $S_8(2)$, $O_8^+(2)$,
and ${}^3D_4(2)$, but also to their subgroups $O_8^-(2)$, $A_{10}$, $A_9$, 
$L_2(17)$, and perhaps others.
In each case, one of the subgroups is centralized by a $3D$ or $3E$ element and the
other is not. Hence we obtain another class of each of these six
groups in $G.S_3$, and in each case the class splits into three classes in $G$.
Note however that this does not apply to the subgroups $S_4(4)$ and $L_2(16)$,
which are normalized by the outer automorphism of $F_4(2)$.

There are four further isomorphism types of simple subgroups 
of $F_4(2)$ with order divisible by $13$,
namely ${}^2F_4(2)'$, $L_2(25)$, $L_4(3)$ and $L_3(3)$.
It is easy to see that there is a unique class in each case,
except possibly $L_3(3)$, which is a subgroup of both ${}^2F_4(2)$
and $L_4(3)$. But both these copies of $L_3(3)$
lie in the centralizer of an outer automorphism
of $F_4(2)$, so they are indeed conjugate. Since the $13$-element
is self-centralizing in $G$, it is easy to write down the normalizers
in each case. They are $2\times {}^2F_4(2)$, $2\times L_2(25)\udot 2$,
$2\times L_4(3){:}2$, and $2\times L_3(3){:}2$.

The remaining simple subgroups with order divisible by $7$ are
$A_8$, $A_7$, $L_2(8)$, $L_2(7)$, $S_6(2)$, $U_3(3)$.
Finally there are the simple groups with order divisible by no prime
bigger than $5$, namely $A_5$, $A_6$, $U_4(2)$.
We make no attempt at a complete classification in these cases.

\subsection{Groups centralizing an inner involution}
\label{C2A}
The centralizer in $G.S_3$ of a $2C$ involution is soluble.
The centralizer of a $2B$ involution in $G.S_3$ has shape
$2^8.2^{16}{:}(S_6(2)\times S_3)$, and contains $7A$ elements.
As subgroups containing $7A$ elements turn out to be relatively
easy to classify, we shall not need to consider this case much further here.

The centralizer of a $2A$ involution is $2^{1+20}{:}U_6(2){:}S_3$, and
contains $7B$ elements.
Now in $2^{1+20}{:}U_6(2)$ we see a subgroup $2^{1+20}{:}L_2(8){:}3$, and the
representation of $L_2(8)$ on $2^{20}$ is the direct sum of the Steinberg module
and two copies of the natural module. Since the latter has $1$-dimensional
$1$-cohomology over the field of order $8$, there are in total $64$ conjugacy
classes of $L_2(8)$ in $2^{1+20}{:}L_2(8)$. One of these has normalizer
$S_3\times L_2(8){:}3$, and three more have normalizer $2\times L_2(8){:}3$.
The remainder are fused into $20$ classes of groups all with normalizer
$2\times L_2(8)$. Together these $24$ classes of $L_2(8)$ account for the
full structure constant of $1/6+3(1/2)+20(3/2)=95/3$. Under the action of the
outer automorphism group $S_3$, the three with normalizer $2\times L_2(8){:}3$
are fused into one class, while the last $20$ appear to be fused into
orbits of sizes $3+3+2+6+6$, with normalizers either $2^2\times L_2(8)$
or $2\times L_2(8){:}3$ or $2\times L_2(8)$.

\section{Simple subgroups of $Fi_{22}$}
\label{F22sub}
In $G.S_3$ there is a class of subgroups $Fi_{22}$, each with normalizer
$Fi_{22}{:}2$, and splitting into three classes in $G$.
In $Fi_{22}{:}2$ there is a unique class of subgroups $O_7(3)$.
Since $O_7(3)$ is not a subgroup of $F_4(2)$ or of any of the
$3$-element centralizers, it follows
that the normalizer in $G.S_3$ is $O_7(3){:}2$, and that there are
three classes of $O_7(3)$ in $G$.

In $O_7(3){:}2$ there is a unique class of $G_2(3)$, and the normalizer
in $G.S_3$ of any such $G_2(3)$ is therefore $G_2(3){:}2$, contained
in $Fi_{22}{:}2$. In $G_2(3)$ there is a unique class of $L_2(13)$,
and again the normalizer in $G.S_3$ of such an $L_2(13)$ is
$L_2(13){:}2$, contained in $Fi_{22}{:}2$.

\section{Eliminating other isomorphism types of simple subgroups}
\label{isolist}
First note that the largest order of an element of $G$ is $35$.
Since $L_2(q)$ has an element of order $(q+1)/2$ we can disregard $L_2(q)$ whenever
$q>70$. Similarly, $L_3(q)$ has an element of $(q^2+q+1)/3$ and $U_3(q)$ has an
element of order $(q^2-q+1)/3$ so we can disregard $L_3(q)$ and $U_3(q)$,
and also $G_2(q)$, when $q>10$. Also,
$L_4(q)$, $U_4(q)$ and $S_4(q)$ contains elements of order $(q^2+1)/2$, so can be
disregarded when $q>8$.
All remaining simple groups of small enough order are explicitly given in the Atlas list
\cite[pp. 239--242]{Atlas},
and, using Lagrange's Theorem and CFSG, we obtain
 in addition to the known subgroups above, just $16$ more possibilities,
as follows:
\begin{itemize}
\item $A_{13}$, $A_{14}$, $L_2(19)$, $L_2(27)$, $L_2(49)$, $L_2(64)$, $L_3(9)$, $L_4(4)$,
\item $S_4(8)$, $S_6(3)$, $G_2(4)$, $Sz(8)$, $J_1$, $J_2$, $J_3$, $Suz$.
\end{itemize}
We eliminate these as follows:
\begin{itemize}
\item $L_2(49)$ has elements of order $25$.
\item $L_2(64)$ has elements of order $65$, and is contained in $S_4(8)$.
\item $L_3(9)$ has elements of order $91$.
\item $L_4(4)$ has elements of order $85$.
\item $S_6(3)$ has elements of order $36$.
\item $J_1$ contains $19{:}6$, whereas the Sylow $19$-normalizer in $G$ is $19{:}9$.
\item $J_2$ contains a triple cover $3\udot A_6$, but the centralizers of elements of 
order $3$ in $G$ are either soluble, in the case $3C$, or direct products
$3\times U_6(2)$ or $3\times O_8^+(2){:}3$ in the cases $3A$ and $3B$.
This argument also eliminates $G_2(4)$ and $Suz$, which contain $J_2$.
\item $A_{13}$ can be generated by taking a subgroup $A_5\times A_8$,
restricting to $A_5\times A_5\times 3$, and extending to another $A_5\times A_8$
normalizing the other $A_5$ factor. But the $5$-centralizer is $5\times A_8$, and is
contained in $O_{10}^-(2)$, so the entire construction takes place inside $O_{10}^-(2)$,
and therefore cannot generate $A_{13}$. This argument eliminates also $A_{14}$.
\item The elements of order $7$ and $13$ in $G$ are rational. The elements of order $5$
have fixed space of dimension $7$ in the action of 
the triple cover $3\udot G$ on the $27$-dimensional module over $\mathbb F_4$.
Using the Brauer character table for $Sz(8)$ in characteristic $2$, 
in \cite{ABC}, it is easy to check that 
these properties cannot be achieved in any
restriction of the Brauer character to $Sz(8)$.
\item $L_2(19)$ contains elements of order $10$, so the involutions are in class $2A$
or $2B$; and elements of order $9$, so the elements of order $3$ are in $3C$. 
But the structure constants of type $(2A,3C,5A)$ and $(2B,3C,5A)$ are zero, so 
there is no $A_5$ of this type in $G$. Since $L_2(19)$ is a subgroup of $J_3$,
this argument eliminates also $J_3$.
\item $L_2(27)$ contains elements of order $14$, so either $2A$- and $7B$-elements,
or $2B$- and $7A$-elements. All such $(2,3,7)$ structure constants are zero except for
$(2B,3A,7A)$, where the value is $1/480$. This implies every $L_2(27)$ has
non-trivial centralizer, which is impossible.
\end{itemize}

\section{Subgroups isomorphic to $A_5$}
\label{A5}
There is a well-known character formula for the number of ways a given element
in a specified conjugacy class $C_3$ is the product of an element $x$ from class $C_1$ and
an element $y$ from class $C_2$. Dividing this by the order of the centralizer of an element
of class $C_3$ gives the symmetrized structure constant 
$$\xi(C_1,C_2,C_3)=\frac{|C_1|.|C_2|.|C_3|}{|G|^2}\sum_{\chi\in Irr(G)} \frac{\chi(x)\chi(y)\chi(y^{-1}x^{-1})}{\chi(1)},$$
which can also be interpreted as the sum over conjugacy classes of such triples $(x,y,xy)$
of the inverse of the order of the centralizer of $(x,y)$.
Putting this together with the fact that
$$\langle x,y,z\mid x^2=y^3=z^5=1,xy=z\rangle$$
is a presentation for $A_5$, leads to the well-known method of
determining $A_5$ subgroups using structure constants for triples of type $(2,3,5)$.

We calculate the structure constants using GAP \cite{GAP}. It turns out that
just $5$ of the $9$ structure constants of type $(2,3,5)$ in $G$ are non-zero,
as follows:
\begin{itemize}
\item $\xi(2B,3A,5A)=1/720$.
\item $\xi(2B,3B,5A)=59/1440=1/20160+1/224+1/64+1/48$.
\item $\xi(2C,3A,5A)=1/48$.
\item $\xi(2C,3B,5A)=9/32=1/32+1/16+3/16$.
\item $\xi(2C,3C,5A)=1$
\end{itemize}
In particular, every $A_5$ centralizes an inner or outer automorphism,
and hence its normalizer lies in one of the known maximal subgroups.

A classification of the subgroups isomorphic to $A_5$
up to conjugacy is more difficult. We 
use knowledge of the subgroups isomorphic to $A_5$ in the Monster \cite{Anatomy1}.
Now the elements of classes $3A,3B,3C,5A$ lift to Monster classes
$3A,3A,3B,5A$ respectively, so the Monster $A_5$s that live in $2^2.G.S_3$
are as follows:
\begin{itemize}
\item Monster type $2A,3A,5A$, normalizer $(A_5\times A_{12}){:}2$.
\item Monster type $2B,3A,5A$, normalizer $(A_5\times 2M_{22}{:}2).2$.
\item Monster type $2B,3B,5A$, normalizer $S_6.2\times M_{11}$. 
\end{itemize}
Since there is a unique class of $2^2$ in $M_{11}$, and the normalizer is $S_4$,
we obtain a single class of $A_5$ in $G$ of the third type, and the
normalizer in $G.S_3$ is $S_3\times S_5$. The outer automorphism of order $3$ 
centralizing this $A_5$ cannot be in $3E,3F,3G$, so must lie in $3D$.

Turning now to $A_{12}$, the involutions of Monster class $2A$ are the ones
of cycle types $2^6$ and $2^2,1^8$. Hence the $2A^2$ subgroups are of the 
following types:
\begin{itemize}
\item Moving $4$ points: $(12)(34)$, $(13)(24)$.
\item Moving $6$ points: $(12)(34)$, $(12)(56)$.
\item $(12)(34)(56)(78)(9X)(ET)$, $(13)(24)$.
\item $(12)(34)(56)(78)(9X)(ET)$, $(13)(24)(57)(68)(9E)(XT)$.
\end{itemize}
In the third case, there is only an $S_2$ of automorphisms, not $S_3$.
The $2^2$ centralizer in $S_{12}$ is $2^2\times 2^4{:}S_4$, so there is a
unique type of such $A_5$ in $G.S_3$, with normalizer
$(A_5\times 2^3{:}S_4.2){:}2$, and splitting into three classes in $G$.
The attributable structure constant is $6/2^7.3=1/64$.
In the first case, the $A_5$ centralizes $A_8$, and this is a case
we know about, contributing $1/20160$ to the structure constant, and having type $(2B,3B)$.
In the second case, the $A_5$ centralizes $S_6$, and this is another
case we know about, contributing $1/720$ to the structure constant,
and having type $(2B,3A)$. In the last case, 
the normalizer of the $2^2$ in $S_{12}$ is $2^6{:}(S_3\times S_3)$,
so giving an $A_5$ in $G.S_3$ with normalizer $(A_5\times 2^4{:}3^2{:}2){:}2$,
contributing $6/2^5.3^2=1/48$ to the structure constant.

Finally we need to consider the case $2.M_{22}.2$, which is rather more
intricate. This group has $6$ classes of involutions, which in
Atlas notation are $-1A$, $+2A$, $-2A$, $+2B$, $-2B$ and $2C$.
These fuse into $2.HS.2$ classes $-1A$, $-2A$, $+2A$, $2C$, $2C$, 
and $2D$ respectively. Hence they fuse in $HN$,
and also in the Monster, to $2A,2B,2A,2A,2A,2B$
respectively.
We therefore need to classify $2^2$ subgroups consisting of elements
of classes $-1A$, $-2A$, and $\pm 2B$ only.
We obtain the following cases:
\begin{itemize}
\item $-1A,+2B,-2B$, with centralizer $2^2\times 2^3{:}L_3(2)$,
and normalized only by an involution that also normalizes the $A_5$.
\item two of type $-2A,-2A,-2A$ with normalizers in $M_{22}$ of
shape $2^4.3^2.2$ and $2^4.S_3$ respectively.
\item one each of types $-2A,\pm 2B,\pm 2B$.
\end{itemize}
The first case gives an $A_5$ with normalizer 
$(2^3{:}L_3(2)\times A_5){:}2$, contributing $6/1344=1/224$
to the structure constant, and being of type $2B,3B$.

The remaining contributions to the structure constants
are as follows.
The case given above as $2^4.3^2.2$ in $M_{22}$
becomes $2.2^4.S_3.S_3$ in $2.M_{22}.2$, and factoring out
the normal $2^2$ as well as taking the $S_3$ off the top
appears to leave us with a structure constant of $1/48$.
The other case of this type differs by a factor of $3$,
so gives $1/16$.

In the remaining cases we have types $(-2A,+2B,+2B)$ and
$(-2A,-2B,-2B)$ which are swapped by an element
normalizing the $A_5$: hence these cases do not extend to
$S_5$ in $G.S_3$. There is also the type $(-2A,+2B,-2B)$
which does extend to $S_5$ in $G.S_3$.
The centralizer of the $2^2$ in the first two cases would seem to be
$2^2\times 2^3S_4$ in $2.M_{22}.2$, but there is also an element
swapping the two elements of type $2B$ in the $4$-group,
hence giving a total contribution
to the structure constant of $1/32$. In the other case
the centralizer in $2\udot M_{22}.2$ is smaller by a factor of $6$, leaving us
with $3/16$.

This accounts exactly for the structure constants, and also
gives all the class fusions, except that there is an
ambiguity as to which of the two contributions of $1/48$
to attribute to $\xi(2B,3B,5A)$ and which to $\xi(2C,3A,5A)$.
To summarize:
\begin{theorem}
There are ten classes of $A_5$ in $G.S_3$, becoming $18$ classes in $G$.
The normalizers in $G.S_3$ are as follows.
\begin{itemize}
\item $S_3\times S_5$, type $(2C,3C)$, contributing $1$ to the structure constant;
\item $(A_5\times(A_8\times 3){:}2){:}2$, type $(2B,3B)$, contributing $1/20160$;
\item $S_5\times S_6\times S_3$, type $(2B,3A)$, contributing $1/720$;
\item $(A_5\times 2^4{:}3^2{:}2){:}2$, type $2C,3A)$, contributing $1/48$;
\item $(A_5\times [2^7.3]).2$, type $(2B,3B)$, splitting into $3$ classes in $G$,
contributing $1/64$;
\item $S_5\times 2^3{:}L_3(2)$, type $(2B,3B)$, splitting into $3$ classes in $G$,
contributing $1/224$;
\item $(A_5\times [2^5.3^2]).2$, type $(2B,3B)$, contributing $1/48$;
\item $(A_5\times 2^3S_4).2$, type $(2C,3B)$, contributing $1/32$;
\item $A_5 \times [2^6.3]$, type $(2C,3B)$, splitting into $3$ classes in $G$,
contributing $1/16$;
\item $(A_5\times [2^4]).2$, type $(2C,3B)$, splitting into $3$ classes in $G$,
contributing $3/16$.
\end{itemize}
\end{theorem}
\section{Hurwitz groups}
\label{Hurwitz}
A Hurwitz group is a group generated by an element $x$ of order $2$ and an element
$y$ of order $3$ such that $xy$ has order $7$. Among the simple groups we are
interested in here, $L_2(7)$, $L_2(8)$, $L_2(13)$, ${}^3D_4(2)$ and $Fi_{22}$ are
Hurwitz groups. Another group which turns up is the non-split extension
$2^3\udot L_3(2)$. To aid in identification of the group generated by $x$ and $y$,
we compute a \emph{fingerprint}, consisting of the orders of the elements
$$s,st,sst,ssts,sstst,ssstst,ssststt$$
where $s=xy$ and $t=xyy$. The last of these words is the only one that does not
necessarily have the same order as its reverse, so we shall often augment the
fingerprint with the order of $sssttst$ in order to distinguish a pair 	$(x,y)$
from its reciprocal $(x^{-1},y^{-1})=(x,yy)$. Note that this fingerprint is
significantly redundant when $s$ has order $7$, but is more discriminating
for larger orders.

We pre-compute the fingerprints for the $(2,3,7)$ generators of all the above groups.
In the list below we adopt an obvious shorthand for the
four reciprocal pairs of Hurwitz generators for $Fi_{22}$. All generating pairs not so marked are 
either self-reciprocal, or automorphic to their reciprocals.
The generating pairs in $Fi_{22}$ are of type $(2C,3D)$ while those in ${}^3D_4(2)$
are of type $(2B,3B,7D)$.
$$\begin{array}{lrrrrrrr}
L_2(7):&7&4&4&7&4&3&3\cr
L_2(8):&7&9&9&7&7&9&7\cr
{}^3D_4(2):&7&14&14&7&21&13&18\cr
L_2(13):&7&6&6&7&7&7&13\cr
&7&7&7&7&6&6&13\cr
&7&13&13&7&3&7&7\cr
Fi_{22}:&7&13&13&7&21&30&8/12\cr
&7&18&18&7&12&24&12/12\cr
&7&20&20&7&15&8&9/13\cr
&7&21&21&7&12&21&30\cr
&7&24&24&7&30&15&13/20
\end{array}$$

In order to classify the Hurwitz subgroups of $G$, we relate them to
Hurwitz subgroups of the Monster, as follows.
First, note that
all involutions in $G$ lift to involutions in $2^2.G$, which is a subgroup of the Monster,
and all elements of order $7$ in $G$ fuse to class $7A$ in the Monster. 
Moreover, the classes $3A$, $3B$, $3C$ in $G$ fuse to $3A$, $3A$, $3B$
respectively in the Monster. It is therefore possible
to use Norton's analysis \cite{Anatomy1}
of the corresponding structure constants in the Monster, to
write down a complete list of Hurwitz subgroups of $G$. 
Of help also is the corresponding analysis in the Baby Monster \cite{genusB}.

First we compute the structure constants of type $(2,3,7)$ in $G$.
Exactly $7$ of the $18$ structure constants are non-zero, as follows:
$$\begin{array}{ll}
\xi(2A,3A,7A)=1/20160 \qquad& \cr
\xi(2B,3A,7A)=1/480 \qquad &\xi(2B,3B,7B)=15/56\cr
\xi(2C,3A,7A)=3/64 \qquad &\xi(2C,3B,7B)=43/8\cr
\xi(2C,3C,7A)=11/3 \qquad &\xi(2C,3C,7B)=329/3
\end{array}$$

\subsection{The $2A,3A,7A$ case}
We have already discussed the $L_3(2)$ of type $(2A,3A,7A)$, which has 
centralizer $L_3(4)$, and embeds in $A_5\times A_8$ permuting the $8$ points
transitively. 
In particular, since the noralizer is $(L_3(2)\times L_3(4)){:}2$, 
this accounts for the full structure constant.
\subsection{The $2B,3A,7A$ case}
The group $A_8$ also contains two classes of $L_3(2)$ acting on
$7$ points, and therefore of type $(2B,3A,7A)$. The visible normalizer is just
$A_5\times L_3(2)$, but the structure constant is $\xi(2B,3A,7A)=1/480$, so this cannot be the
whole normalizer. Since the centralizer lies inside $L_3(4)$, 
the only possibility is $2^4A_5\times L_3(2)$, which exactly accounts
for the structure constant. In particular, this $L_3(2)$ does not extend to
$L_3(2){:}2$ inside $G$.
\subsection{The $2C,3C,7A$ case}
To account for $\xi(2C,3C,7A)=11/3$, note that $1$ is attributable to ${}^3D_4(2)$, and $1$
to $L_3(2)$, and $3/2$ to an $L_2(8)$ with normalizer $2\times L_2(8)$, and $1/6$ to an $L_2(8)$
with normalizer $S_3\times L_2(8){:}3$.
\subsection{The $2C,3C,7B$ case}
To account for $\xi(2C,3C,7B)=329/3$, note that $54$ is attributable to $Fi_{22}$,
since there are three classes of $Fi_{22}$ in $G$, each with two
(automorphic) classes of generating triples of each of $9$ types. A further $3$
is attributable
to three more classes of ${}^3D_4(2)$, and $18$ to three classes of $L_2(13)$,
and $3$ to three classes of $L_3(2)$;
the remainder $95/3=1/6+21(3/2)$ is attributable to the seven classes of $L_2(8)$ in $G.S_3$
identified in Sections~\ref{C3A} and \ref{C2A}.

\subsection{The $2B,3B,7B$ case}
There are three classes of $L_3(2)$ with normalizer $(L_3(2)\times L_3(2)){:}2$,
accounting for an amount $3/168=1/56$ of the structure constant, leaving
an amount $14/56=1/4$.

\subsection{The $2C,3A,7A$ case}
Since the structure constant is $3/64$, every such group is centralized
by an involution, necessarily of class $2B$.

\subsection{The $2C,3B,7B$ case}
In this case an explicit computer search found $36$ such triples
of elements in $G$, of which $6$ generate $L_3(2)$ and
$30$ generate $2^3\udot L_3(2)$.
Now the latter group has four classes of $(2,3,7)$ generating triples.

It would seem likely therefore that the structure constant of $5\frac38$
should be attributed as $4$ for $2^3\udot L_3(2)$ and $1\frac38$
for at least two classes of $L_3(2)$ in $G.S_3$. In any case, since the
total structure constant is less than $6$, it follows that any $L_3(2)$
of this type has non-trivial centralizer in $G.S_3$, and hence its 
normalizer cannot be maximal.

\subsection{Conclusion}
In this section we have proved the following theorem.
Apart from the calculation of structure constants, the proof is
mostly computer-free. However, we used
computational methods to check the results, by collecting
large numbers of $(2,3,7)$ triples and using statistical analysis
of frequencies to show that the above allocation of
structure constants to isomorphism types of groups is the only
plausible one.
\begin{theorem}
\begin{itemize} \item In $G.S_3$ there is a unique class of each of 
$L_2(13)$ and $Fi_{22}$, two classes of ${}^3D_4(2)$, and nine classes
of $L_2(8)$, as listed in Tables~\ref{simpleslarge} and \ref{simples7}.
The ones whose normalizers are maximal are just $Fi_{22}$ (in $G$ and $G.2$)
and one class of ${}^3D_4(2)$ (in $G.3$ and $G.S_3$).
\item  Every $L_3(2)$ in $G$ has non-trivial centralizer in $G.S_3$.
Only one class of $L_3(2)$ has maximal normalizer (in all of $G$, $G.2$,
$G.3$ and $G.S_3$).
\end{itemize}
\end{theorem}

\section{Classifying $(2,3,11)$ triples}
\label{11triples}
We next analyse the structure constants of type $(2,3,11)$ in a similar manner.
First we pre-compute the fingerprints for simple groups we know to be generated
by such a triple, including $L_2(11)$, $M_{12}$, $A_{11}$, $A_{12}$, $U_6(2)$,
$O_{10}^-(2)$, $Fi_{22}$. Note that the groups
 $2^{10}{:}L_2(11)$ and $2\udot U_6(2)$ are also generated by $(2,3,11)$ triples.
The fingerprints for the smaller groups are listed in Table~\ref{2311small}.
\begin{table}
\caption{\label{2311small}Some fingerprints of type $(2,3,11)$}
$$\begin{array}{llrrrrrrr}
L_2(11) & 2A,3A & 11&5&5&6&3&11&11\cr
A_{11}&2B,3C & 11&8&11&8&12&11&11/21\cr
 && 11&12&9&14&9&9&8\cr
&&11&14&10&11&11&14&8\cr
&&11&20&12&11&5&8&12/12\cr
&&11&21&8&11&12&11&6/12\cr
O_{10}^-(2)&2D,3F &11&9&15&8&12&35&30\cr
&&11&12&30&18&17&12&12/17\cr
&&11&17&12&6&18&18&17\cr
&&11&20&17&24&21&9&11/18\cr
&&11&21&21&24&5&18&30\cr
&&11&24&12&30&15&21&12/17\cr
M_{12} & 2A,3A & 11&6&10&6&10&6&8\cr
&2B,3B & 11&6&6&11&6&11&8\cr
A_{12} & 2C,3C & 11&14&9&20&11&35&4/9\cr
&&11&21&35&35&10&8&11\cr
U_6(2) & 2C,3C & 11&9&18&12&7&12&18\cr
 && 11&12&11&8&12&8&18
\end{array}$$
\end{table} 
The fingerprints of the $89$ triples which generate $Fi_{22}$ are listed in
Table~\ref{F22gen2311}.
\begin{table}
\caption{\label{F22gen2311}The $89$ fingerprints of type $(2,3,11)$ for $Fi_{22}$}
$$\begin{array}{lrrrrrrr}
&11&6&11&18&22&18&13/21\cr
&11&6&14&8&13&8&9/15\cr
&11&7&15&12&13&12&11/16\cr
&11&8&12&10&22&20&13/21\cr
&11&9&10&11&12&13&15/24\cr
&11&9&11&13&24&8&11/12\cr
&11&9&11&13&24&11&12/16\cr
&11&9&21&11&10&8&22\cr
&11&9&21&30&12&22&12/14\cr
&11&9&24&22&16&22&16/21\cr
&11&10&12&22&13&13&18/18\cr
&11&10&13&11&22&11&20/24\cr
&11&10&14&18&20&22&13/15\cr
&11&10&20&16&13&22&12/18\cr
&11&12&20&14&6&20&12/22\cr
&11&12&21&11&21&12&20\cr
&11&12&21&30&11&13&9/9\cr
&11&12&22&13&30&13&14/24\cr
&11&12&22&14&12&11&13/21\cr
&11&13&9&20&8&20&8/22\cr
&11&13&11&20&22&9&13/14\cr
&11&13&13&11&12&14&8/21\cr
&11&13&13&11&18&8&13/21
\end{array}
\begin{array}{lrrrrrrr}
&11&13&13&13&24&12&12/15\cr
&11&13&13&18&10&12&9/24\cr
&11&13&13&22&13&20&12/18\cr
&11&13&21&20&16&18&18/21\cr
&11&13&24&15&8&13&12/13\cr
&11&13&24&30&12&21&11/12\cr
&11&14&8&11&9&12&13/13\cr
&11&18&9&13&15&14&12/16\cr
&11&18&12&9&13&13&14/22\cr
&11&18&12&22&16&9&14/15\cr
&11&18&14&18&9&8&13/15\cr
&11&18&15&15&14&10&10/13\cr
&11&18&22&20&11&22&18/24\cr
&11&20&12&18&9&16&16/22\cr
&11&20&22&8&13&14&13/30\cr
&11&21&7&24&18&21&12\cr
&11&24&8&16&15&30&14/22\cr
&11&24&9&22&13&22&10/10\cr
&11&24&11&12&16&11&8/9\cr
&11&24&13&11&14&22&5/7\cr
&11&30&16&18&13&20&11/24\cr
&11&30&21&8&14&30&10/18\cr
&11&30&24&11&11&30&13/14
\end{array}$$
\end{table}
(For practical purposes, especially for distinguishing
between a triple and its reciprocal,
these fingerprints were then extended by the orders of $sssststt$ and
$sssststtt$ 
to give extra discriminating power.)

We calculate the following structure constants 
of type $(2,3,11)$ in $G$ (all others are zero):
\begin{itemize}
\item $\xi(2B,3B,11A/B)=1/6$ each. This is accounted for by $L_2(11)$.
\item $\xi(2B,3C,11A/B)=1$ each. 
Such a subgroup must be centralized by a
non-trivial automorphism, and 
one of the two types of $(2,3,11)$
generating tripes for the subgroup $M_{12}$ of $O_{10}^-(2)$
accounts for the full structure constant.
\item $\xi(2C,3B,11A/B)=37/2$ each.
\item $\xi(2C,3C,11A/B)=1650$ each.
\end{itemize}

\subsection{The $(2C,3B,11)$ case}
%
The known simple subgroups make the following contributions to each
structure constant of $18\frac12$: 
\begin{itemize}
\item $1$ from the other type of $(2,3,11)$ generators
for the subgroup $M_{12}$, with normalizer $M_{12}\times 3$ in $G.S_3$;
\item $3$ from $A_{12}$, with normalizer $(A_{12}\times 3){:}2$ in $G.S_3$;
\item $1$ from $U_6(2)$, with normalizer $S_3\times U_6(2){:}S_3$ in $G.S_3$.
\end{itemize}

Now the $1$-cohomology of $U_6(2)$ on the $20$-dimensional module
is $2$-dimensional, so there are four classes of $U_6(2)$ in $2^{20}{:}U_6(2)$,
which fuse in $2^{20}{:}U_6(2){:}S_3$ into one class of $U_6(2)$ with normalizer
$U_6(2){:}S_3$, and one class with normalizer $U_6(2){:}2$. The former lifts to
$2\times U_6(2){:}S_3$, and the corresponding normalizer in $G.S_3$ is
$S_3\times U_6(2){:}S_3$. On the other hand, there is a subgroup
$2\udot U_6(2)$ inside $Fi_{22}$, and therefore the other case must lift
to $2\udot U_6(2).2$. This group does not centralize any element of order $3$
in $G.S_3$, so this is the full normalizer. Hence the three classes
of $2\udot U_6(2)$ in $G$ account for a contribution of $9$ to each
of the two 
$(2,3,11)$ structure constants.

Also within $2^{1+20}{:}U_6(2)$ there is a subgroup $2^{1+20}{:}L_2(11)$,
which contains three conjugacy classes of $2\times 2^{10}{:}L_2(11)$, permuted by
the $S_3$ of outer automorphisms. Each $(2,3,11)$ generating triple for
$L_2(11)$ lifts to $4$ triples of type $(2,3,11)$ in $2^{10}{:}L_2(11)$,
since the elements of orders $2$ and $3$ each centralize a $6$-space in the $2^{10}$.
One of these triples generates the complementary $L_2(11)$, 
which is of type $(2B,3B)$ in $G$, while the other
three must generate the whole of $2^{10}{:}L_2(11)$.
Extending to $2^{1+20}{:}U_6(2){:}S_3$ we have a subgroup
$2^{1+20}{:}(U_5(2)\times 3){:}2$ and hence a subgroup
$2^{1+20}{:}(L_2(11)\times 3){:}2$. Altogether, therefore, the
contribution to each $(2,3,11)$ structure constant is $9/2$.

Altogether we have accounted for an amount $1+3+1+9+4\frac12=18\frac12$,
that is the whole structure constant.

\subsection{The $(2C,3C,11)$ case}
In this case, it turns out that most of the triples generate the whole
of $G$, and a computer search to collect and identify these. 
The total number of $(2C,3C,11A/B)$ fingerprints collected was $562$. 
Of these, $1$ generates $L_2(11)$, and $8$ generate $A_{11}$, while
$9$ generate $O_{10}^-(2)$. These subgroups all centralize an outer automorphism
of order $3$, so each fingerprint contributes $1$ to each structure constant.
The remaining $544$ fingerprints correspond to generators of either $Fi_{22}$ or $G$,
and in either case contribute $3$ to each structure constant. Thus we have accounted
for the full structure constant of $18+3\times544=1650$.

As noted above, there are $89$ fingerprints for generators of $Fi_{22}$,
listed in Table~\ref{F22gen2311}.
The $455$ fingerprints for
$(2,3,11)$ triples that generate $G$ are listed in the Appendix.

\subsection{Conclusion}
In this section we have proved the following result. The proof
depends crucially on the computational analysis of the $(2,3,11)$
triples of type $(2C,3C)$.
\begin{theorem}
 There is a unique class of each of the groups $O_{10}^-(2)$,
$A_{12}$, $A_{11}$, $M_{12}$ and $U_6(2)$, and two classes of
$L_2(11)$ in $G.S_3$. In every case the centralizer in $G.S_3$
is non-trivial, and the only cases in which the normalizer
is maximal are $O_{10}^-(2)$ and $U_6(2)$ (in all of $G$,
$G.2$, $G.3$ and $G.S_3$).
\end{theorem}

\section{Classifying $(2,3,13)$ triples}
\label{13triples}
Groups generated by $(2,3,13)$ triples include
$L_2(13)$, $L_2(25)$, $L_3(3)$, $L_4(3)$, $G_2(3)$,
${}^2F_4(2)'$, $^3D_4(2)$, $F_4(2)$ and $Fi_{22}$. Of these,
we have already classified $L_2(13)$, ${}^3D_4(2)$ and
$Fi_{22}$ using Hurwitz generators. The group
$3^3{:}L_3(3)$ is also generated by ($8$ classes of)
$(2,3,13)$ triples, and so is the maximal subgroup
$3^{3+3}{:}L_3(3)$ of $O_7(3)$, a fact
which initially caused me a significant amount of
difficulty in accounting for the structure
constants of type $(2C,3C,13)$.

We pre-compute the fingerprints of the smaller simple groups above,
and list them in Table~\ref{2313small}.
There are also $109$ fingerprints for generators of type $(2,3,13)$ for $Fi_{22}$,
which are of $Fi_{22}$-type $(2C,3D,13)$ and are listed in Table~\ref{F22gens2313};
as well as $261$ fingerprints for $(2,3,13)$ generators of $F_4(2)$
that are of $F_4(2)$-type $(2D,3C,13)$ and are listed in Table~\ref{F42gens2D3C13}. 

\begin{table}
\caption{\label{2313small}Some fingerprints of type $(2,3,13)$}
$$\begin{array}{llrrrrrrr}
L_2(13) && 13 &7&6&7&7&2&7\cr
L_2(25) && 13&4&13&13&13&3&6\cr
&& 13 &12&6&13&3&13&13\cr
&& 13&13&12&5&13&13&13\cr
L_3(3)&&13 &4&8&13&8&3&13\cr
&& 13&6&13&13&8&13&13\cr
{}^2F_4(2)'& 2A,3A & 13 & 5 & 12 & 8 &8& 13 & 10\cr
L_4(3) & 2A,3D & 13 & 5 & 12 & 20&20&13&13\cr
  && 13&5&20&6&6&13&6\cr
G_2(3)& 2A,3C & 13&13&8&7&6&12&8\cr
&2A,3E & 13&9&13&12&7&9&8\cr
&&13&12&12&7&8&8&7/8\cr
{}^3D_4(2)& 2B,3A & 13 &21&12&13&6&21&13\cr
&&13&28&18&21&6&13&12/18\cr
& 2B,3B & 13&7&18&18&21&8&12/14\cr
&&13&9&12&13&21&13&12/14\cr
&&13&18&28&13&14&18&12/21\cr
&&13&21&28&18&28&21&12\cr
Fi_{22}& 2C,3C & 13&12&12&22&16&13&11/30\cr
&&13&18&14&22&9&20&13/20\cr
&&13&21&18&11&9&15&16/16
\end{array}$$
\end{table}
\begin{table}
\caption{\label{F22gens2313}Fingerprints of type $(2,3,13)$ for $Fi_{22}$}
$$\begin{array}{rrrrrrrr}
13&6&9&8&21&8&14/22\cr
13&6&16&18&13&18&11/18\cr
13&6&18&30&13&30&16\cr
13&6&30&13&30&13&16\cr
13&7&12&12&21&4&13\cr
13&7&12&22&22&8&13/22\cr
13&7&13&11&10&8&9/14\cr
13&7&13&12&24&4&13/13\cr
13&8&7&13&24&13&18/18\cr
13&8&9&8&11&20&11/14\cr
13&8&9&8&11&21&8/9\cr
13&8&16&8&8&13&7/8\cr
13&8&21&13&12&13&13/21\cr
13&8&22&13&9&22&13/20\cr
13&9&11&10&15&9&11/12\cr
13&9&15&12&22&21&13/13\cr
13&12&7&13&24&11&11/24\cr
13&12&9&13&16&13&11/16\cr
13&12&9&16&16&30&8/12\cr
13&12&10&13&14&13&9/24\cr
13&12&11&9&11&13&11/21\cr
13&12&11&11&13&22&8/24\cr
13&12&11&18&13&18&10/24\cr
13&12&13&12&12&21&9/11\cr
13&12&14&13&11&12&13/30\cr
13&12&14&13&16&22&11/24\cr
13&12&14&16&10&10&9/22\cr
13&12&14&21&21&13&18/18\cr
13&12&16&13&11&20&12/12
\end{array}\qquad
\begin{array}{rrrrrrr}
13&12&16&13&16&22&12/12\cr
13&12&16&22&9&13&8/22\cr
13&12&21&13&8&13&13/22\cr
13&13&11&9&24&22&21/21\cr
13&13&13&11&24&13&9/9\cr
13&14&8&9&11&11&10/10\cr
13&14&8&15&13&11&9/22\cr
13&14&11&16&6&18&12/13\cr
13&14&12&18&30&18&21/22\cr
13&14&13&21&12&13&11/13\cr
13&14&14&11&8&7&16/22\cr
13&14&16&12&24&14&21/22\cr
13&14&16&22&21&7&9/12\cr
13&14&22&9&8&11&12/30\cr
13&14&22&9&8&14&11/30\cr
13&15&13&7&16&9&14/22\cr
13&18&12&13&9&21&8/12\cr
13&18&14&22&8&30&12/18\cr
13&21&12&9&12&12&30\cr
13&21&12&9&18&11&13/18\cr
13&21&12&18&13&12&6/10\cr
13&21&14&12&18&22&14/18\cr
13&21&14&16&13&11&9/30\cr
13&21&18&15&21&24&11\cr
13&21&20&10&12&18&13/20\cr
13&21&20&22&13&13&10/11\cr
13&30&10&8&21&10&13/22\cr
13&30&13&21&16&21&14/18\cr
\end{array}$$
\end{table}
\begin{table}
\caption{\label{F42gens2D3C13}$(2D,3C,13)$ generators for $F_4(2)$}
$$\begin{array}{rrrrrrr}
13&6&10&13&16&13&28\cr
13&6&13&12&21&12&17/18\cr
13&6&16&28&12&28&17\cr
13&6&16&30&30&30&16/17\cr
13&6&17&18&28&18&24\cr
13&6&24&18&24&18&28\cr
13&8&12&10&13&20&6\cr
13&8&13&9&28&30&17/24\cr
13&8&13&30&18&21&12/24\cr
13&8&17&17&28&28&17/21\cr
13&8&17&21&9&17&21/30\cr
13&8&20&17&13&8&6/18\cr
13&8&21&28&24&24&28/30\cr
13&8&24&28&30&28&13/21\cr
13&8&30&28&21&24&18/21\cr
13&10&10&12&12&15&16\cr
13&10&12&8&8&15&12\cr
13&10&13&10&13&13&21/30\cr
13&10&14&16&12&8&16\cr
13&10&24&24&16&30&14\cr
13&12&9&18&17&24&21/30\cr
13&12&12&17&21&13&20/28\cr
13&12&12&24&17&21&17/20\cr
13&12&13&12&16&21&18/24\cr
13&12&13&12&28&12&13/30\cr
13&12&14&16&21&17&24/24\cr
13&12&16&13&30&9&12/28\cr
13&12&16&16&12&21&12/21\cr
13&12&16&18&13&17&17/21\cr
13&12&17&12&12&12&12/20\cr
13&12&17&16&16&13&9/18\cr
13&12&17&20&20&18&9/24\cr
13&12&17&21&30&18&16/17\cr
13&12&18&12&21&17&12/16\cr
13&12&18&13&21&12&21/28\cr
13&12&18&17&30&20&13/18\cr
13&12&18&28&24&13&10/12\cr
13&12&20&18&12&30&17/24\cr
13&12&21&9&20&30&12/18\cr
13&12&21&17&12&13&21/30\cr
13&12&21&18&20&21&16/21\cr
13&12&24&28&12&30&18/28\cr
13&12&28&17&18&15&13/18\cr
13&12&30&17&30&18&12/24\cr
13&12&30&21&28&24&12/16
\end{array}
\qquad
\begin{array}{rrrrrrr}
13&13&13&13&13&10&24\cr
13&13&13&18&18&13&20/28\cr
13&13&17&21&12&18&16/21\cr
13&13&21&14&10&17&16\cr
13&13&24&12&12&17&24/24\cr
13&14&12&28&18&12&8/28\cr
13&14&17&28&21&17&17/20\cr
13&14&18&14&14&21&12/21\cr
13&15&16&10&30&15&21/21\cr
13&15&18&21&24&14&13/18\cr
13&16&9&24&15&17&24/28\cr
13&16&10&9&12&20&17/20\cr
13&16&12&13&21&16&12/28\cr
13&16&12&17&28&28&24/28\cr
13&16&13&8&8&13&10\cr
13&16&13&9&12&21&12/13\cr
13&16&13&13&28&12&21/24\cr
13&16&15&9&17&10&17/30\cr
13&16&17&16&21&28&12/24\cr
13&16&18&12&30&16&18/30\cr
13&16&18&13&17&21&10/28\cr
13&16&20&13&16&12&28\cr
13&16&21&8&10&12&12/28\cr
13&16&21&20&21&28&18/28\cr
13&16&21&30&18&17&16/18\cr
13&16&24&21&14&24&12/28\cr
13&16&24&30&21&30&18/30\cr
13&16&30&12&12&21&16/28\cr
13&17&8&8&15&17&12/18\cr
13&17&9&20&21&12&10/28\cr
13&17&9&28&28&12&21/24\cr
13&17&12&16&6&21&17/28\cr
13&17&13&10&28&18&24/24\cr
13&17&13&13&24&16&13/21\cr
13&17&15&21&24&8&12/24\cr
13&17&17&17&12&13&12/18\cr
13&17&17&28&24&12&17/21\cr
13&17&18&15&18&30&20/21\cr
13&17&20&13&15&18&8/9\cr
13&17&21&17&20&20&24/30\cr
13&17&21&28&12&20&17/17\cr
13&17&21&30&12&16&12/13\cr
13&17&24&17&17&28&17/24\cr
13&17&24&17&30&21&9/13\cr
13&17&24&20&30&13&17/21
\end{array}$$
\end{table}
\begin{table}
$$
\begin{array}{rrrrrrr}
13&17&24&24&24&12&13/14\cr
13&17&24&24&30&8&21/24\cr
13&17&28&21&30&21&12/21\cr
13&17&28&28&12&21&17/21\cr
13&17&30&18&18&21&8/16\cr
13&18&8&10&10&18&24/28\cr
13&18&12&12&18&21&12/18\cr
13&18&16&12&21&21&13/30\cr
13&18&18&12&14&20&20/21\cr
13&18&28&28&18&24&9/9\cr
13&20&9&17&16&10&12/24\cr
13&20&13&10&21&13&12/24\cr
13&20&13&13&13&12&17\cr
13&20&16&8&10&13&28\cr
13&20&20&18&24&28&17/18\cr
13&20&21&24&24&17&24/30\cr
13&20&28&30&18&10&20/21\cr
13&21&8&16&21&12&24\cr
13&21&12&14&20&15&9/15\cr
13&21&12&15&10&12&13\cr
13&21&12&24&28&17&15/24\cr
13&21&15&24&12&12&18\cr
13&21&16&17&17&24&10/13\cr
13&21&17&17&18&12&24/24\cr
13&21&21&12&24&8&21/30
\end{array}\qquad
\begin{array}{rrrrrrr}
13&21&24&17&30&17&21/21\cr
13&21&24&30&20&10&24\cr
13&24&12&15&12&24&16/30\cr
13&24&12&16&20&15&15/28\cr
13&24&14&20&30&13&14/18\cr
13&24&16&21&14&17&28/30\cr
13&24&17&17&30&14&12/13\cr
13&24&18&13&18&12&16/17\cr
13&24&18&17&24&12&18/24\cr
13&24&20&20&16&18&12/24\cr
13&24&24&13&28&21&17/30\cr
13&24&30&30&21&21&21/28\cr
13&28&12&12&13&14&18/20\cr
13&28&12&20&12&12&8/16\cr
13&28&13&24&8&20&24/30\cr
13&28&18&20&18&16&13/17\cr
13&28&21&24&21&18&13/16\cr
13&28&21&24&24&20&12/30\cr
13&28&24&9&18&20&17/20\cr
13&28&28&12&13&20&8/21\cr
13&30&12&17&21&24&13/18\cr
13&30&13&24&16&17&17/18\cr
13&30&21&13&21&12&24/24\cr
13&30&21&17&28&18&24/30\cr
13&30&21&30&24&21&18/18\cr
\end{array}$$
\end{table}

In fact the non-zero $(2,3,13)$ structure constants in $F_4(2)$ are
\begin{itemize}
\item $\xi(2C,3C,13A)=5$, of which $1$ is from ${}^2F_4(2)'$ and $4$ is from $L_4(3)$;
\item $\xi(2D,3A,13A)=\xi(2D,3B,13A)=3$, from generators of ${}^3D_4(2)$;
\item $\xi(2D,3C,13A)=552$, of which $14$ is from ${}^3D_4(2)$, and $2$ is from
$L_3(3)$, and $8$ is from $3^3{:}L_3(3)$, 
and $6$ is from $L_2(25)$, 
leaving $522$ for
$261$ automorphic pairs of generators for $F_4(2)$.
\end{itemize}

There are only four non-zero structure constants of type $(2,3,13)$ in $G$,
that is
\begin{itemize}
\item $\xi(2B,3C,13A)=15$;
\item $\xi(2C,3A,13A)=3$;
\item $\xi(2C,3B,13A)=63$;
\item $\xi(2C,3C,13A)=9658$.
\end{itemize}
The first three of these can be easily accounted for by known subgroups.
To increase the reliability of the results, however, we also computed explicitly
triples of group elements of these types.
\subsection{The $(2B,3C,13)$ case}
An explicit search of the $2B,3C$ case yields three fingerprints with $xy$ of order $13$,
namely those for ${}^2F_4(2)'$ and $L_4(3)$.
There are three known conjugacy classes of ${}^2F_4(2)'$ in $G$, one in each of the
three copies of $F_4(2)$, and each with normalizer ${}^2F_4(2)$. Each is centralized
by an outer automorphism of order $2$, so in total this accounts for an amount $3$
of the structure constant. 
There are also
three known conjugacy classes of subgroups $L_4(3)$ in $G$, 
again one in each of the three copies
of $F_4(2)$. In this case, the outer automorphism group of $L_4(3)$ is $2^2$, of which
only $2$ is realised in $G$. It follows that the amount of structure constant
accounted for here is $2\times2\times 3=12$. Hence we have accounted
for the full structure constant of $3+12=15$.

\subsection{The $(2C,3A,13)$ case}
\label{2C3A13}
An explicit search of the $(2C,3A)$ case yields three fingerprints
with $xy$ of order $13$, equal to the three fingerprints for ${}^3D_4(2)$
generators of type $(2B,3A,13)$.
Since there is a subgroup ${}^3D_4(2)$ of this type, it accounts for the
full structure constant.

\subsection{The $(2C,3B,13)$ case}
\label{2C3B13}
A explicit search of the $(2C,3B)$ case
yields $12$ fingerprints where $xy$ has order $13$,
equal to the pre-computed fingerprints for $Fi_{22}$ (six fingerprints
of type $(2C,3C,13)$),
$G_2(3)$ (three of type $(2A,3E,13)$), and ${}^3D_4(2)$ (three of type $(2B,3A,13)$).

There are three conjugacy
classes in $G$ of self-normalizing subgroups $Fi_{22}$, which account for an amount
$6\times 6=36$ of the structure constant.
Also there are
three classes of self-normalizing subgroups $G_2(3)$, accounting for an amount
$3\times 6=18$ of the structure constant. 
There is also a generating triple of $G_2(3)$-type $(2A,3C)$,
which turns out to be of type $(2C,3C)$ in $G$ (see below).

Recall that
there are two classes of ${}^3D_4(2)$ inside $F_4(2)$, differing in the
fusion of
$7$-elements to $F_4(2)$, and hence to $G$. One class has been already found
above, in Subsection~\ref{2C3A13},
generated by $(2C,3A,13)$-triples. These groups lie in centralizers of 
outer automorphisms of order $3$. The three triples in the present case therefore
generate a ${}^3D_4(2)$ with normalizer ${}^3D_4(2){:}3\times 2$ in $G.S_3$,
so account for an amount $3\times 3=9$ of the structure constant.

This accounts for the full amount $36+18+9=63$ of structure constant.

\subsection{The $(2C,3C,13)$ case}
\label{2C3C13}
The case $(2C,3C)$ is analyzed as follows.
We collected $1599$ distinct
fingerprints, of which $1213$ generated $G$, hence accounting for
an amount $7278$ of the structure constant. There were $261$ generating $F_4(2)$, and
$109$ generating $Fi_{22}$, accounting for a further $2220$, making a running total
of $9498$. Two copies of ${}^3D_4(2)$ account for $7+21$ of the structure constant,
while single copies of $G_2(3)$, $L_2(13)$ and $L_3(3)$ account for a further $18$,
for a running total of $9544$. The subgroup $L_2(25)$ also accounts for $18$,
making a total of $9562$. Out of the total of $9658$, therefore,
there is still $96$ to account for.

Now there are three classes of $3^3.3^3.L_3(3)$ in $G$. The total $(2,3,13)$
structure constant in these groups is $3.4.3.3=108$. Of this, $12$ comes
from triples generating (three classes of)
$L_3(3)$, while $24$ is attributable to (three classes of) $3^3{:}L_3(3)$
and the remaining $72$ comes from generators for the whole group.
Hence we exactly account for the remaining structure constant of $24+72=96$.

\subsection{Conclusion}
In this section we have completely classified subgroups isomorphic to
$L_3(3)$, $L_4(3)$, $G_2(3)$, $L_2(25)$, $F_4(2)$ and ${}^2F_4(2)'$, as well as
verifying the earlier results for $L_2(13)$, ${}^3D_4(2)$ and $Fi_{22}$.

\begin{theorem}
In $G.S_3$ there is a unique class of each of $G_2(3)$, $L_4(3)$,
$F_4(2)$, ${}^2F_4(2)'$, $L_2(25)$ and $L_3(3)$. 
In the case of $F_4(2)$, the normalizer in $G$ and $G.2$ is maximal,
and there are three conjugacy classes. In all other cases the
normalizer is not maximal in any of $G$, $G.2$, $G.3$ or $G.S_3$.
\end{theorem}

\section{Classifying $(2,3,17)$ triples}
\label{17triples}
Known subgroups with $(2,3,17)$ generators include
$L_2(16)$, $O_8^-(2)$, $O_{10}^-(2)$, $S_8(2)$,
$L_2(17)$, $F_4(2)$, and $2^8{:}O_8^-(2)$ and
$2^{8+16}{:}O_8^-(2)$.
We precompute the fingerprints in Table~\ref{2317small}.
\begin{table}
\caption{\label{2317small}Some fingerprints of type $(2,3,17)$}
$$\begin{array}{llrrrrrrrrr}
O_{10}^-(2)&2D,3E&17&30&33&15&35&33&18\cr
&&17&12&33&11&18&30&11/21\cr
&&17&17&20&30&6&8&9\cr
O_8^-(2)&2C,3B&17&21&9&9&7&6&8/21\cr
&&17&15&17&12&5&9&17\cr
&&17&21&7&17&10&10&30\cr
&2C,3C&17&8&30&12&8&21&10/17\cr
&&17&10&21&17&21&12&10/17\cr
&&17&15&10&17&21&17&6/10\cr
&&17&17&17&8&15&30&7\cr
&&17&17&30&8&7&6&9\cr
&&17&21&12&17&10&17&17/30\cr
L_2(16)&&17&17&17&5&17&15&17\cr
&&17&15&17&17&17&17&5\cr
L_2(17)&&17&9&9&8&9&9&9\cr
S_8(2)&2D,3D&17&6&30&12&21&12&21\cr
&&17&6&30&24&21&24&30\cr
&2E,3D&17&7&12&24&15&4&17\cr
&&17&8&17&15&18&12&8/21\cr
&&17&9&20&15&6&12&12\cr
&&17&12&12&18&20&30&10/17\cr
&&17&12&12&24&20&10&17/20\cr
&&17&15&20&30&15&15&10/12\cr
&2F,3D&17&6&21&24&12&24&20\cr
&&17&7&30&14&21&12&12/17\cr
&&17&9&14&17&17&15&5/17\cr
&&17&12&12&9&12&30&12\cr
&&17&12&18&14&30&14&10/12\cr
&&17&17&15&18&21&30&12/17\cr
&&17&20&30&17&14&17&17/17\cr
&&17&30&24&12&12&15&4/21\cr
F_4(2) & 2C,3C & 17&6&30&12&24&12&13/16\cr
&&17&6&30&12&28&12&12
\end{array}$$
\end{table}

The structure constants we need to account for are the following:
\begin{itemize}
\item $\xi(2B,3C,17A+B)=13+13=26$;
\item $\xi(2C,3A,17A+B)=3+3=6$;
\item $\xi(2C,3B,17A+B)=35+35=70$;
\item $\xi(2C,3C,17A+B)=7614+7614=15228$.
\end{itemize}

\subsection{Triples inside $F_4(2)$}
We shall also need information about $(2,3,17)$ generating triples for
$F_4(2)$. To determine these, we compute the structure constants in $F_4(2)$
as follows.
\begin{itemize}
\item $\xi(2C,3C,17A+B)=5+5$, of which $2+2$ are $(2D,3D)$ generators for 
two classes of $S_8(2)$ and $3+3$ are generators for $F_4(2)$ whose
fingerprints are given above;
\item $\xi(2D,3A,17A+B)=3+3$, of which $2+2$ are $(2C,3B)$ generators for
one class of $O_8^-(2)$ and $1+1$ are generators for one class of $L_2(16)$;
\item $\xi(2D,3B,17A+B)=3+3$, similarly accounted for by the other
class of $O_8^-(2)$ and of $L_2(16)$;
\item $\xi(2D,3C,17A+B)=420+420$, of which $10+10$ are the other generators
for both classes of $O_8^-(2)$, and $24+24$ are the other generators for
both classes of $S_8(2)$, and $2+2$ are generators for $L_2(17)$.
\end{itemize}
This leaves $3+3$ of type $(2C,3C)$ and $384+384$ of type $(2D,3C)$ all of
which must generate the whole of $F_4(2)$. The fingerprints for the former
are given in Table~\ref{2317small} and for the latter
 in Table~\ref{F42gen2D3C}.

\begin{table}
\caption{\label{F42gen2D3C}$(2D,3C,17)$ generators for $F_4(2)$}
$$\begin{array}{rrrrrrrr}
&17&4&28&17&28&3&10/21\cr
&17&6&9&21&12&21&16/21\cr
&17&6&17&20&21&20&9/20\cr
&17&6&24&21&20&21&14/17\cr
&17&7&18&21&28&8&17/28\cr
&17&7&28&12&15&4&12\cr
&17&8&9&21&30&12&8/18\cr
&17&8&10&16&28&13&12/17\cr
&17&8&13&9&13&17&16/21\cr
&17&8&13&13&16&14&13/16\cr
&17&8&13&18&13&28&8/17\cr
&17&8&13&21&13&21&10/30\cr
&17&8&14&12&17&18&8/18\cr
&17&8&16&17&17&20&13/21\cr
&17&8&16&18&28&20&17/24\cr
&17&8&17&18&21&28&17/24\cr
&17&8&17&24&30&12&12/16\cr
&17&8&21&14&13&21&18/24\cr
&17&8&21&16&30&12&9/12\cr
&17&8&21&21&13&24&12/18\cr
&17&8&21&21&14&21&13/13\cr
&17&8&24&18&13&21&13/16\cr
&17&8&28&17&12&21&21/30\cr
&17&8&30&9&6&24&17\cr
&17&8&30&18&12&21&28/30\cr
&17&9&12&9&20&24&16\cr
&17&9&14&20&12&15&18\cr
&17&9&17&12&24&17&12/24\cr
&17&9&17&28&20&18&17\cr
&17&9&24&28&20&9&28\cr
&17&9&28&16&9&18&30\cr
&17&10&10&13&24&30&21/24\cr
&17&10&12&21&17&8&9/18\cr
&17&10&17&24&16&18&17/28\cr
&17&10&20&28&14&30&18/21\cr
&17&10&21&14&28&16&21/30\cr
&17&10&24&21&18&18&12/21\cr
&17&10&30&12&24&18&14/24\cr
&17&12&8&24&21&12&10\cr
&17&12&9&17&12&12&13/17\cr
&17&12&9&18&17&17&12/13\cr
&17&12&12&10&20&13&12/16\cr
&17&12&12&13&15&17&16/18\cr
&17&12&12&13&21&21&17/24\cr
&17&12&12&17&16&21&8/28
\end{array}\qquad
\begin{array}{rrrrrrrr}
&17&12&12&30&21&13&9/20\cr
&17&12&13&9&8&20&12/24\cr
&17&12&13&13&21&16&16/18\cr
&17&12&13&28&12&18&21/24\cr
&17&12&16&12&28&12&13\cr
&17&12&16&16&9&9&21/24\cr
&17&12&16&20&18&18&13/16\cr
&17&12&17&24&28&21&12/21\cr
&17&12&18&13&20&30&12\cr
&17&12&18&13&30&17&13/14\cr
&17&12&18&16&16&17&12/15\cr
&17&12&18&18&9&17&16/17\cr
&17&12&18&21&12&17&16\cr
&17&12&18&28&21&28&18/18\cr
&17&12&21&13&13&30&15/16\cr
&17&12&21&17&13&17&12/18\cr
&17&12&21&17&24&16&28/30\cr
&17&12&21&18&9&24&18/18\cr
&17&12&24&30&17&12&13/30\cr
&17&12&28&17&17&13&24/24\cr
&17&12&28&18&12&17&20/24\cr
&17&12&28&21&12&24&12/24\cr
&17&12&28&24&28&17&13\cr
&17&12&28&28&9&21&9/13\cr
&17&12&30&17&21&12&12/18\cr
&17&13&18&9&14&21&8/13\cr
&17&13&20&24&18&8&17/21\cr
&17&13&21&12&14&28&12/15\cr
&17&13&21&13&30&10&18/30\cr
&17&14&12&13&16&12&30\cr
&17&14&12&24&12&12&13\cr
&17&14&13&9&12&14&21/28\cr
&17&14&20&24&14&18&12\cr
&17&14&28&24&28&10&18/24\cr
&17&14&30&16&12&14&24\cr
&17&14&30&28&21&24&9/30\cr
&17&15&18&17&9&28&17/24\cr
&17&15&18&17&17&18&9/15\cr
&17&15&18&28&12&21&12/30\cr
&17&15&21&20&18&24&16/24\cr
&17&15&21&20&20&24&21/28\cr
&17&15&30&13&18&13&21/30\cr
&17&16&8&14&21&12&30/30\cr
&17&16&12&9&12&18&9/13\cr
&17&16&14&16&24&21&18/24
\end{array}$$
\end{table}
\begin{table}
$$\begin{array}{rrrrrrrr}
&17&16&14&21&21&24&12/24\cr
&17&16&17&14&14&24&9/10\cr
&17&16&17&17&21&21&15/18\cr
&17&16&20&9&24&9&9/12\cr
&17&16&20&30&21&14&12/28\cr
&17&16&21&12&17&18&18/21\cr
&17&16&21&13&21&16&30/30\cr
&17&16&21&16&28&24&9/17\cr
&17&16&28&20&17&13&6/21\cr
&17&17&6&14&24&24&8\cr
&17&17&8&16&17&20&14\cr
&17&17&12&14&13&12&12\cr
&17&17&21&16&28&14&14\cr
&17&17&21&28&18&8&10/28\cr
&17&17&21&28&24&21&17/18\cr
&17&17&24&16&16&18&24\cr
&17&17&28&12&9&18&24/24\cr
&17&17&28&12&18&16&20/3\cr
&17&18&9&12&13&14&17/21\cr
&17&18&9&20&20&21&12/15\cr
&17&18&12&18&12&17&10/28\cr
&17&18&13&17&18&28&13/17\cr
&17&18&15&17&12&12&14/20\cr
&17&18&17&9&13&14&13/14\cr
&17&18&18&20&12&16&20/24\cr
&17&18&21&12&18&13&12/24\cr
&17&18&24&16&8&8&18/18\cr
&17&20&13&24&18&18&24/30\cr
&17&20&14&18&24&21&17/30\cr
&17&20&15&30&21&9&16/24\cr
&17&20&16&15&17&30&21/24\cr
&17&20&17&16&18&12&30/30\cr
&17&20&18&18&12&13&15/21\cr
&17&20&24&12&28&12&17/18\cr
&17&20&24&17&13&12&13/20\cr
&17&20&28&30&16&28&13/24\cr
&17&20&30&12&18&24&16/20\cr
&17&21&8&30&12&14&17/21\cr
&17&21&9&12&20&17&16/17\cr
&17&21&13&12&12&16&18/24\cr
&17&21&16&24&28&13&17/21\cr
&17&21&17&13&30&28&18/30\cr
&17&21&18&18&16&28&16/28\cr
&17&21&20&17&20&13&10/13\cr
&17&21&20&18&24&24&15/28\cr
&17&21&20&28&21&13&17/30\cr
&17&21&21&8&14&16&12/30
\end{array}\qquad
\begin{array}{rrrrrrrr}
&17&21&21&13&14&12&18/28\cr
&17&21&24&13&24&24&18/30\cr
&17&21&28&24&28&21&20/24\cr
&17&21&30&12&28&18&15/18\cr
&17&21&30&17&13&14&12/16\cr
&17&24&8&20&24&18&18/21\cr
&17&24&8&24&18&20&18/30\cr
&17&24&8&24&18&21&10/18\cr
&17&24&9&30&15&13&17/28\cr
&17&24&10&16&13&9&18/21\cr
&17&24&12&12&12&18&16/17\cr
&17&24&12&16&14&24&12/13\cr
&17&24&12&21&18&8&13/20\cr
&17&24&13&16&16&21&12/20\cr
&17&24&13&24&13&24&12/21\cr
&17&24&13&24&21&16&21/24\cr
&17&24&15&21&13&17&12/12\cr
&17&24&16&12&12&17&21/21\cr
&17&24&16&17&17&18&18/28\cr
&17&24&17&12&24&30&18/21\cr
&17&24&17&13&24&20&13/15\cr
&17&24&17&18&9&18&16/18\cr
&17&24&17&18&12&24&13/18\cr
&17&24&17&28&12&8&13/30\cr
&17&24&17&30&28&12&13/24\cr
&17&24&18&13&21&30&13/16\cr
&17&24&18&13&24&24&13/21\cr
&17&24&18&16&13&20&12/12\cr
&17&24&20&17&17&20&7/12\cr
&17&24&20&21&20&15&12/24\cr
&17&24&20&24&8&24&17/17\cr
&17&24&20&30&24&12&24/30\cr
&17&24&21&13&24&13&24/30\cr
&17&24&21&18&24&17&13/17\cr
&17&24&21&30&21&17&18/20\cr
&17&24&24&21&24&18&12/16\cr
&17&24&28&12&28&24&18/30\cr
&17&24&30&13&12&20&13/20\cr
&17&28&9&28&21&18&17/18\cr
&17&28&12&17&24&20&13/21\cr
&17&28&12&18&30&12&18/24\cr
&17&28&13&13&24&18&12/21\cr
&17&28&13&9&24&21&13/17\cr
&17&28&15&21&30&12&17/24\cr
&17&28&17&13&12&20&18/20\cr
&17&28&20&15&13&24&14/24\cr
&17&28&20&17&24&21&14/24
\end{array}$$\end{table}
\begin{table}
$$\begin{array}{rrrrrrrr}
&17&28&21&16&30&18&20/24\cr
&17&28&30&12&30&12&14/24\cr
&17&30&12&12&24&12&14/18\cr
&17&30&12&21&21&24&17/30\cr
&17&30&12&24&18&17&18/28\cr
&17&30&13&16&18&18&17/21\cr
&17&30&13&20&21&20&17/21\cr
&17&30&13&30&24&12&18/21\cr
&17&30&14&16&14&17&21/24
\end{array}\qquad
\begin{array}{rrrrrrrr}
&17&30&17&18&18&13&16/21\cr
&17&30&18&16&8&12&12/13\cr
&17&30&18&18&12&13&14/18\cr
&17&30&18&24&17&16&15/24\cr
&17&30&21&17&17&9&17/17\cr
&17&30&21&18&17&16&9/21\cr
&17&30&24&12&18&16&10/24\cr
&17&30&28&24&13&16&16/30\cr
&17&30&28&28&16&17&12/17
\end{array}
$$
\end{table}

\subsection{The $(2B,3C,17)$ case}
In the case $(2B,3C)$ we found five fingerprints with $xy$ of order $17$.
The first three are for generators of
$F_4(2)$. The three classes of (self-normalizing) $F_4(2)$ therefore account for
an amount $9+9$ of the structure constant.

The other two cases are generators of type $(2D,3D,17)$ for $S_8(2)$. 
Now $F_4(2)$ contains two
conjugacy classes of $S_8(2)$, which can be distinguished by the class of $7$-elements
they contain. 
In one case, the $S_8(2)$ lies in $O_{10}^-(2)$, so is centralized by a non-inner
automorphism. In the other case there are three conjugacy classes in $G$, fused in $G.3$.
Hence these four classes of $S_8(2)$ account for an amount $2+6=8$ of the
structure constant.  Together these subgroups account for the full
structure constant $26$.

\subsection{The $(2C,3A,17)$ case}
In the case $(2C,3A)$, there are four fingerprints
with $xy$ of order $17$ that generate $O_8^-(2)$. One of the known
classes of $O_8^-(2)$ is centralized by an $S_3$ of outer automorphisms, so each
fingerprint for this group accounts for $1$ of the structure constant, making $4$ altogether.
There are two fingerprints that generate $L_2(16)$. One of the known classes 
of $L_2(16){:}4$ therefore
accounts for the remaining $2$ of the structure constant.

\subsection{The $(2C,3B,17)$ case}
In the case $(2C,3B)$ we found $18$ fingerprints when the order of $xy$ is $17$.
Four are generators of type $(2D,3E)$ for $O_{10}^-(2)$,
four are generators of type $(2C,3B)$ for $O_8^-(2)$, and two
are generators for $L_2(16)$, while the remaining $8$ generate $G$. The
fingerprints for the generators for $G$ are as follows:
$$\begin{array}{rrrrrrrrr}
17&12&33&22&33&22&17/18\cr
17&18&22&12&13&19&13/18\cr
17&28&16&17&16&12&14/24\cr
17&28&20&20&8&18&15/30
\end{array}$$

The generators for $G$ account for an amount $48$ of the 
structure constant. The generators for $O_{10}^-(2)$ account
for an amount $8$. Now there are two classes
of $O_8^-(2)$ in $F_4(2)$, one of which has already been counted in the enumeration
of the $(2C,3A)$ case. The other gives rise to three classes of $O_8^-(2)$
in $G$, each extending to $O_8^-(2){:}2$, and each
centralized by an outer automorphism of order $2$.
Together these account for an amount $12$ of the structure constant.

Finally, the second 
known class of $L_2(16){:}4$
centralizes an outer $S_2$ (but not $S_3$)
and accounts for an amount $2$ of the structure constant.
This accounts for the full amount $48+8+12+2=70$ of the structure constant.

\subsection{The $(2C,3C,17)$ case}
The $(2C,3C)$ case is considerably harder to analyse.
The subgroups which turn out to be generated in this way are
$L_2(17)$, $O_8^-(2)$, $S_8(2)$, $F_4(2)$, as well as $2^8{:}O_8^-(2)$
and $2^{8+16}{:}O_8^-(2)$. The latter two $2$-local subgroups have many
classes of $(2,3,17)$ generating triples, and it is hard to distinguish many of them
using fingerprints of the type we have been using. We found only $86$
distinct fingerprints, but it would appear that the number of distinct
(i.e. non-automorphic) such triples (using an element of $O_8^-(2)$ class  $3C$)
is $200$ in the case of $2^{8+16}{:}O_8^-(2)$, and $30$ in the case
of $2^8{:}O_8^-(2)$.

To prove this, observe that the $2C$ elements in $O_8^-(2)$ centralize
$2^4$ of the $2^8$, while the $3C$ elements centralize just $2^2$. Hence each
$(2C,3C,17)$ triple in $O_8^-(2)$ lifts to four $(2,3,17)$ triples in
$2^8{:}O_8^-(2)$, of which three generate the whole group. Similarly, lifting to
$2^{8+16}{:}O_8^-(2)$ we acquire a factor of $2^6$, so each triple for $O_8^-(2)$
lifts to $64$ triples, of which $60$ generate the whole group. However, there is
an outer automorphism of order $3$ which effectively reduces this number to $20$.
We have therefore $1200$ generating triples, falling into $200$ types under the action
of the outer automorphism group $S_3$. Therefore this contributes $600$ to
each of the two structure constants. Similarly for $2^8{:}O_8^-(2)$ we have
$60$ generating triples, each centralized by an outer automorphism of order $3$,
and swapped in pairs by the outer automorphism of order $2$, so this group contributes
$30$ to each structure constant of $7614$. (Running total so far $630$.)

Now the fingerprint collection found in total $1919$ distinct fingerprints of type
$(2C,2C,17A/B)$ for generators of ${}^2E_6(2)$. However, it is not possible to have an odd number of such fingerprints, as there are equal numbers with $xy$ in $17A$
and in $17B$. (This argument depends crucially on the fact that 
$17A$ and $17B$ are not fused by the outer automorphism group.)
Therefore there must be at least $1920$ such generating triples,
up to automorphisms. These account for an amount
$3\times1920=5760$ of each structure constant. (Running total $6390$.)

Next, we have shown that $F_4(2)$ has $384$ generating
triples of this type, for each of the two classes of $17$-elements. Hence the three
classes of $F_4(2)$ together contribute $3\times 384=1152$ to each structure constant.
(Running total $7542$.)

For each of the groups $L_2(17)$, $O_8^-(2)$ and $S_8(2)$, there are two known
classes in ${}^2E_6(2){:}S_3$, one centralized by an outer $S_3$, the
other centralized by only an outer $2$. The counting is slightly different in each case.
In the case of $L_2(17)$, the outer automorphism of $L_2(17)$ is not realised,
and therefore the unique fingerprint contributes a total of $1+3=4$ to each structure constant.
(Running total $7546$.) In the case of $S_8(2)$, there is no outer automorphism,
and therefore each fingerprint (of which there are $24$) contributes just $2$
to each structure constant. (Running total $7594$.)
Finally, in the case of $O_8^-(2)$ the outer automorphism is realised inside ${}^2E_6(2)$, 
so each
of the $10$ fingerprints again contributes $2$. (Total so far $7614$, of an expected
$7614$.)

The analysis of $(2,3,17)$ triples in this section proves that there is no
subgroup isomorphic to one of $L_2(16)$, $L_2(17)$, $O_8^-(2)$, $O_{10}^-(2)$ or 
$S_8(2)$ 
other than those contained in known maximal subgroups.

\section{Classifying $(2,3,19)$ triples}
\label{19triples}
Analysis of these triples proves only one thing: that there is a unique conjugacy
class of $U_3(8)$. However, there does not seem to be any easier way to prove
this result.
The structure constants that we need to account for are as follows:
\begin{itemize}
\item $\xi(2B,3C,19A+B)=9+9=18$;
\item $\xi(2C,3B,19A+B)=15+15=30$;
\item $\xi(2C,3C,19A+B)=6126+6126=12252$.
\end{itemize}
The expected numbers of fingerprints therefore are respectively $3$, $5$ and $2042$.

In the case $(2B,3C)$ we find three fingerprints, all for triples that generate $G$,
as follows:
$$\begin{array}{rrrrrrr}
19&5&16&28&28&19&24\cr
19&5&30&12&12&19&19\cr
19&5&33&30&15&19&10
\end{array}$$
In the case $(2C,3B)$, similarly, we find five fingerprints,
all for triples that generate $G$, as follows:
$$\begin{array}{rrrrrrr}
19&9&30&13&13&19&8/16\cr
19&18&12&13&10&30&24/30\cr
19&21&24&21&16&19&30
\end{array}$$

In the $(2C,3C)$ case when $xy$ has order $19$, we found $2041$ distinct fingerprints.
Since we expecting $2042$, we looked morely closely, and found that one
fingerprint which appeared to be from a self-reciprocal generating set was in fact
from a pair of mutually reciprocal generating sets, but the fingerprint was not
sufficiently discriminating to pick this up.
Of these $2042$ distinct types of triples, $2041$ generate $G$, so contribute $3$ to each structure constant,
making a total contribution of $6123$ to the structure constant of 6126.
The remaining fingerprint corresponds to subgroups $U_3(8)$. Now the known
subgroup $U_3(8)$ has normalizer $(U_3(8){:}3 \times 3)){:}2$ in $G.S_3$, so
contributes $3$ to each structure constant.
This fully accounts for the structure constant, and proves that there is a
unique class of $U_3(8)$ in ${}^2E_6(2)$.

\section{Status report}
\label{status}
Using computational analysis of $(2,3,n)$ triples for $n=7,11,13,17,19$
we have robust classifications for simple subgroups of the following
isomorphism types
\begin{itemize}
\item $L_2(8)$, $L_2(13)$, ${}^3D_4(2)$, $Fi_{22}$;
\item $L_2(11)$, $M_{12}$, $A_{11}$, $A_{12}$, $O_{10}^-(2)$, $U_6(2)$;
\item $L_2(25)$, $L_3(3)$, $L_4(3)$, $G_2(3)$, ${}^2F_4(2)'$;
\item $L_2(16)$, $L_2(17)$, $O_8^-(2)$, $S_8(2)$, $F_4(2)$;
\item $U_3(8)$;
\end{itemize}
as well as incomplete classifications of $A_5$
and $L_3(2)$
which are nevertheless sufficient for the purposes of
determining maximal subgroups. We can therefore say that we have
 dealt with $23$ of the $39$ cases. The remaining $16$ cases
require other methods, more group theoretic than
character theoretic.

The groups we need to deal with are
\begin{itemize}
\item $A_6$, $A_7$, $A_8$, $A_9$, $A_{10}$;
\item $L_3(4)$, $U_3(3)$, $U_4(2)$, $U_4(3)$, $U_5(2)$;
\item $O_7(3)$, $O_8^+(2)$, $S_4(4)$, $S_6(2)$;
\item $M_{11}$, $M_{22}$.
\end{itemize}
Most of these contain $A_5$, and therefore can be analyzed
using Norton's classification \cite{Anatomy1} of subgroups of the Monster
containing $5A$-type $A_5$. 
To ensure that Norton's methods carry through, we need to be
sure that in every case where the simple group $H$ he is interested
in has a double cover $2\udot H$, he is using a copy of $A_5$ that is
doubly covered in $2\udot H$. This is the case for all the alternating
groups, and $U_4(2)$, and therefore $S_6(2)$, $O_8^+(2)$ and $O_7(3)$;
but not for $L_3(4)$, $U_4(3)$ or $M_{22}$.

Hence the cases that require other methods
are $U_3(3)$, which does not contain $A_5$, and $L_3(4)$,
$M_{22}$ and $U_4(3)$, all of which have proper double
covers in which all involutions lift to involutions.
For all except $L_3(4)$, it is sufficient to use
known results on subgroups of the Baby Monster.
The case $L_3(4)$ is problematic, since there is 
a subgroup $2^2\udot L_3(4)$ in $2^2\udot{}^2E_6(2)$,
and there may potentially be more than one conjugacy class,
and such groups will not be detected by existing work
on the Monster and Baby Monster.
In this last case we need to extend Norton's methods to
include classifications of $2^2\udot L_3(4)$ and $2\times 2\udot L_3(4)$ 
in the Monster, as
well as the simple group $L_3(4)$. 

\section{Using the Monster}
\label{Monster}
The fact that the Monster contains a subgroup $2^2\udot G{:}S_3$ means that
what is known about subgroups of the Monster implies facts about subgroups of $G$.
Of particular interest is Norton's list \cite{Anatomy1}
of simple subgroups of the Monster containing $A_5$ with $5A$-elements. Now all
elements of order $5$ in $G$ fuse to $5A$ in the Monster, and all involutions in
$G$ lift to involutions in $2^2\udot G$, so this effectively deals with almost all
simple subgroups of $G$ which contain $A_5$. The ones which require
extra care are those which have a double cover in which all
involutions lift to involutions. There is no such group on Norton's list, but 
in principle we need to
consider the following possibilities:
\begin{itemize}
\item $L_3(4)$, $M_{22}$, $U_4(3)$, $U_6(2)$, $Fi_{22}$, $F_4(2)$.
\end{itemize}
We shall not consider the cases $U_6(2)$, $Fi_{22}$ and $F_4(2)$, however,
as these have been adequately dealt with by other methods.

In this section we go through the whole of Norton's list, including the
groups we have already dealt with. This is partly in order to provide 
alternative, computer-free, proofs in many cases, but also to demonstrate the
reliability of Norton's work, which is published 
essentially without proof.

\subsection{Six very easy cases}
From Norton's list we pick out first the following,
which are the simple groups that centralize a unique class of $2^2$
in the Monster. In each case we write down the subgroup of
the Monster that is the direct product of the simple group under consideration
with its centralizer (which Norton calls its Monstralizer).

\begin{itemize}
\item $O_{10}^-(2) \times A_4$
\item $O_7(3)\times S_3\times S_3$
\item $A_{12}\times A_5$
\item $A_{11}\times A_5$
\item ${}^2F_4(2)'\times 2\udot S_4$
\item $L_2(25)\times 2\udot S_4$
\end{itemize}
It follows that there is a unique class of the corresponding simple group in $G.S_3$.
In the cases $H\cong O_{10}^-(2)$, $A_{12}$ and $A_{11}$ the centralizer contains $A_4$,
and the normalizers in $G.S_3$ are of the shape $(H\times 3){:}2$. In particular,
there is a single conjugacy class of $H$ in $G$.

In the case $O_7(3)$, the normalizer in $G.S_3$ is $O_7(3){:}2$, and there
are three conjugacy classes of $O_7(3)$ in $G$. 
These are not maximal in $G$, since they are contained in $Fi_{22}$.
However, in $G.2$ the normalizer is $O_7(3){:}2$, which is not contained in
$Fi_{22}{:}2$, so both groups are maximal.

In the cases
$H\cong L_2(25)$ and ${}^2F_4(2)'$, the centralizer is $2\udot S_4$, containing
a unique class of $2^2$. On factoring out by this $2^2$, an extra centralizing
involution appears, and in each case we have normalizer in $G.S_3$ of shape
$2\times H.2$. Again we obtain three conjugacy classes of $H$ in $G$.

For our proof we only require the case $O_7(3)$ and summarize as follows.
\begin{theorem} There is a unique class of $O_7(3)$ in $G.S_3$.
The normalizer is $O_7(3){:}2$, and the class splits into three classes in $G$.
\end{theorem}

\subsection{Six more easy cases}
Next we pick from Norton's list the following five groups, whose centralizers
in the Monster
contain two distinct classes of $2^2$:
\begin{itemize}
\item $S_8(2)\times S_4$
\item $O_8^-(2)\times S_4$
\item $O_8^+(2)\times (3\times A_4).2$
\item $A_{10}\times S_5$
\item $A_9\times (3\times A_5).2$
\end{itemize}
In all these cases, one of the two classes of $2^2$ in the centralizer
is normalized to $S_4$, the other only to $D_8$. Thus we obtain two classes of $H$ in
$G.S_3$, one of which
splits into three classes in $G$.
It is easy to see that the normalizers are as given in
Table~\ref{simpleslarge}, \ref{simples7} or \ref{simples5}.

The case $S_6(2)$, with centralizer $S_4\times S_3$, is similar.
Note that all involutions in $S_4\times S_3$ are in the Monster class $2A$, except for
those of shape $(12)(34)(ab)$, which are in $2B$. There are therefore three types of
$2A$-pure $2^2$, and hence three conjugacy classes of $S_6(2)$ in $G.S_3$.
In $G$, the normalizers are as follows:
\begin{itemize}
\item one class of $S_3\times S_6(2)$, of type $7A$;
\item three classes of $S_3\times S_6(2)$, of type $7B$;
\item six classes of $2\times S_6(2)$, of type $7B$.
\end{itemize}
The normalizers in $G.S_3$ are therefore 
\begin{itemize}
\item $S_3\times S_3\times S_6(2)$, contained in $(3\times O_{10}^-(2)){:}2$;
\item $2\times S_3\times S_6(2)$, contained in $2\times F_4(2)$;
\item $2\times S_6(2)$, contained in $N(2A)$.
\end{itemize}
We remark that the remaining lifts of $S_6(2)$ in $2^{1+10}{:}U_6(2)$ must therefore be
$2\udot S_6(2)$, in three conjugacy classes.

For our proof we need the cases $O_8^+(2)$, $S_6(2)$, $A_{10}$ and $A_9$.

\subsection{Three quite easy cases}
Now consider the following cases, where again there are two
types of $2^2$ in the centralizer in the Monster.
\begin{itemize}
\item $M_{12}\times L_2(11)$
\item $S_4(4)\times L_3(2)$
\item $L_2(16)\times L_3(2)$
\end{itemize}
Here both classes of $2^2$ in the centralizer 
extend to $A_4$. In the case of $M_{12}$, there is an element of the Monster
extending the group to $(M_{12}\times L_2(11)){:}2$. This swaps the two classes
of $2^2$, while also effecting the outer automorphism of $M_{12}$. It follows that
there is a single class of $M_{12}$ in $G.S_3$, with normalizer $3\times M_{12}$,
and this class splits into two classes in $G$.

Similarly in the second case we have $(S_4(4).2 \times L_3(2)).2$. Hence there is
a unique class of $S_4(4)$ in $G$, with normalizer $S_4(4).2\times S_3$ in $G.S_3$,
contained in $(3\times O_{10}^-(2)){:}2$.

In the third case we have $L_2(16){:}4 \times L_3(2)$ instead, so there are
two conjugacy classes of $L_2(16)$ in $G.S_3$, with normalizer $L_2(16){:}4\times S_3$
in each case.

For our proof, we need only the case $S_4(4)$, and summarize as follows.
\begin{theorem}
There is a unique class of $S_4(4)$ in $G.S_3$, and the normalizer is
$S_3\times S_4(4){:}2$, contained in $(3\times O_{10}^-(2)){:}2$.
The class remains a single class in $G$.
\end{theorem}

\subsection{Five tougher groups}
\begin{itemize}
\item $A_7 \times (A_5\times A_5).2.2$
\item $A_8 \times (A_5\times A_4).2$
\item $U_5(2) \times S_3\times A_4$
\item $L_4(3) \times 3^2{:}D_8$
\item $M_{11}\times S_6.2$
\item $M_{11}\times L_2(11)$
\end{itemize}
We do not need the case $L_4(3)$, where we see two classes of $2^2$ in the
centralizer. These are swapped by an outer automorphism of $L_4(3)$
realised in the Monster. Hence there is a unique class of $L_4(3)$ in $G.S_3$.

The $U_5(2)$ centralizer is $S_3\times A_4$. Embedding the latter in $11\times M_{12}$, we see
that the only Monster $2A$-elements are in the $A_4$. Hence there is a unique class of $U_5(2)$
in $G.S_3$, and any $U_5(2)$ centralizes an element of order $3$ in $G$.
It follows that it lies in $S_3\times U_6(2)$, as does its normalizer in $G$.
 
 In the second $M_{11}$ case, the
centralizer in the Monster is $L_2(11)$, whose centralizer is $M_{12}$. The same argument
as for $M_{12}$, therefore, shows that there are just two classes of $M_{11}$
of this type
in $G.S_3$, each with normalizer $3\times M_{11}$, contained in $3\times M_{12}$.

In the other $M_{11}$ case, the Monstralizer is $S_6.2$, contained in $11\times M_{12}$.
But the involutions in the $A_6$ lie in $M_{12}$-class $2B$, and therefore class $2B$ in
the Monster. Hence there is no pure $2A^2$ subgroup in $S_6.2$, so this case
does not arise in $G$.

We next take the group $A_8$ with centralizer $(A_5\times A_4){:}2$
in the Monster. From the embedding of the latter in $A_{12}$ we see that the only $2A$-elements
are in one of the factors $A_4$ or $A_5$. Hence there are exactly two classes of $A_8$
in $G$, and the normalizers in $G$ are respectively $(A_8\times A_5){:}2$ and
$(A_8\times A_4){:}2$. Both groups lie inside $2^{8}{:}O_8^-(2)$ in
$O_{10}^-(2)$.

Finally we take the $A_7$ with centralizer $(A_5\times A_5).2.2$. In the latter
group, the elements of Monster class $2A$ either lie in one of the two $A_5$ factors,
or swap the two factors. Hence there is a unique class of pure $2A$-type $2^2$.
Therefore there is a unique class of $A_7$ in $G$, and the normalizer is
$(A_7\times A_5){:}2$.

From this list we need the cases $A_7$, $A_8$, $M_{11}$, and $U_5(2)$,
and summarize as follows:
\begin{theorem}
\begin{itemize}
\item In $G.S_3$ there is a unique class of $A_7$; each $A_7$ has 
normalizer $(A_7\times (A_5\times 3){:}2){:}2$;
\item in $G.S_3$ there is a unique class of $U_5(2)$, with normalizer $S_3\times (3\times U_5(2)){:}2$;
\item in $G.S_3$ there are two classes of $M_{11}$, each with normalizer $3\times M_{11}$;
\item in $G.S_3$ there are two classes of $A_8$; in one case each $A_8$ has normalizer
$(A_8\times (A_5\times 3){:}2){:}2$, and in the other $(A_8\times (A_4\times 3){:}2){:}2$. 
\end{itemize}
\end{theorem}

\subsection{The last two}
The remaining items on Norton's list are:
\begin{itemize}
\item $A_6 \times (A_6\times A_6).2.2$
\item $A_6 \times 2.L_3(4).2$
\item $A_6 \times M_{11}$
\item $U_4(2) \times (A_4\times S_3\times S_3).2$
\end{itemize}
In the cases $A_6$ and $U_4(2)$, we have some difficulty in
getting the complete list of conjugacy classes.
However, it is quite straightforward to show that there is
no maximal subgroup which is the normalizer of an $A_6$ or $U_4(2)$.

For example, every $2^2$ in $(A_4\times S_3\times S_3).2$
centralizes a further involution, and therefore every $U_4(2)$ in $G$ centralizes an involution.
Similarly, it is easy to see that in the first two $A_6$ cases in the list,
the centralizer of any $2^2$ in the $A_6$-centralizer is larger than the $2^2$ itself.

In the third $A_6$ case, the centralizer is $M_{11}$, which contains a unique
conjugacy class of $2^2$,
whose normalizer is $S_4$. 
It follows that there is a unique class of such $A_6$ in $G.S_3$,
centralizing an $S_3$ of outer automorphisms. Its normalizer therefore
lies in the normalizer of $O_{10}^-(2)$.

\subsection{Conclusion}
In this section we have dealt with the twelve cases $O_7(3)$, $O_8^+(2)$,
$A_{10}$, $A_9$, $S_4(4)$, $S_6(2)$, $U_5(2)$, $U_4(2)$, $M_{11}$,
$A_6$, $A_7$, $A_8$. This leaves just the four cases $U_3(3)$, $M_{22}$,
$U_4(3)$ and $L_3(4)$, where we use properties of the Baby Monster as well.

\section{Using the Baby Monster}
\label{Baby}
\subsection{The case $M_{22}$}
It is shown in \cite{moreB} that there is a unique class of $M_{22}$ containing $5A$-elements
in the Baby Monster. The normalizer is $S_5\times M_{22}{:}2$. Only the transpositions
in $S_5$ fuse to class $2A$ in the Monster. Hence every $M_{22}$ in $G$ centralizes $S_3$
in $G.2$, so centralizes a $3A$ element. It follows that
there are three conjugacy classes of $S_3\times M_{22}$ in $G$, 
lying inside $S_3\times U_6(2)$, and extending to a single class of
$S_3\times M_{22}{:}2$ in $G.S_3$.

\subsection{The case $U_3(3)$}
It is shown in \cite{maxB} that every $U_3(3)$ in the Baby Monster is conjugate in the
Monster to the one 
with centralizer 
$(2^2\times 3^2{:}Q_8){:}S_3$ in the Monster.
This centralizer contains four classes of involutions, with centralizers 
$(2^2\times 3^2{:}Q_8){:}2$, $(2^2\times Q_8){:}S_3$, $2\times SD_{16}$, and
$2^2\times S_3$ respectively. The first must be of Monster type $2A$, since it
gives rise to the subgroup $3^2{:}Q_8\times U_3(3){:}2$ of $G$. 

In any case, no $2^2$ is self-centralizing in $(2^2\times 3^2{:}Q_8){:}S_3$, and
indeed every such $2^2$ centralizes at least a group $2^3$, so every
$U_3(3)$ in $G$ centralizes an involution.

\subsection{The case $U_4(3)$}
Every $U_4(3)$ in $G$ must lift to $2\times 2\udot U_4(3)$ in $2^2\udot G$, and
therefore lifts to $2\times U_4(3)$ in one of the three copies of the Baby Monster,
containing one of the three copies of $2\udot G$. Now it is shown in \cite[Theorem 11.3]{moreB}
that any $U_4(3)$ in the Baby Monster has non-trivial centralizer.
Looking at the proof in more detail, we see that any $2\udot U_4(3)$ in the Monster
has centralizer which is the intersection of two copies of $2\times S_6$ in
$2\udot L_3(4){:}2_2$. But the action has rank $3$ and it is easy to see that
the intersections are $2\times 3^2{:}D_8$ and $2^2\times D_{8}$. In particular,
every $2^2$ in either of these possibilities centralizes a further involution. Hence
every $U_4(3)$ in $G$ centralizes an involution.

\subsection{The case $L_3(4)$}
This last case is problematical because there is an embedding of $2^2\udot L_3(4)$ in
$2^2\udot G$, and hence the enumeration of subgroups $L_3(4)$ in the Baby Monster
and the Monster is not in itself sufficient to deal with the problem. However, we can modify the
argument used in \cite[Theorem 11.2]{moreB}.
Note however that although each of the groups 
$L_3(4)$ and $2\udot L_3(4)$
can be generated by two copies of $A_6$ intersecting in $3^2{:}4$, such that
$3^2{:}Q_8$ interchanges these two copies of $A_6$,
this is no longer true in $2^2\udot L_3(4)$, where the intersection is
only $3^2{:}2$. 

Now there are three different types of $A_6$ that need to be considered. The first 
has centralizer $(A_6\times A_6).2.2$ in the Monster, and contains elements of
Monster class $2A$. These necessarily map to $2A$ elements in $G$. Now the only
non-zero structure constants of type $(2A,4,7)$ in $G$ are
\begin{itemize}
\item $\xi(2A,4A,7A)=1/20160$, fully accounted for by the $L_3(2)$ with normalizer
$(L_3(2)\times L_3(4)){:}2$.
\item $\xi(2A,4H,7A)=1/480$. Consider the subgroup $2^3{:}L_3(2)$ of
$A_8$ in $A_5\times A_8$. The centralizer lies between $A_5$ and $L_3(4)$, and has order
at least $480$, so is $2^4{:}A_5$. Hence the normalizer is $2^3{:}L_3(2)\times 2^4{:}A_5$,
fully accounting for the structure constant.
\item $\xi(2A,4L,7A)=1/192$. There must be such a group inside the $3B$-centralizer,
hence in $O_8^+(2)$. Moreover, its centralizer is in $L_3(4)$ but contains no
elements of order $5$, so has order at most $192$, and therefore exactly $192$.
But it cannot be $L_3(4)$, since $L_3(4)$ is not a subgroup of $O_8^+(2)$.
\end{itemize}

Next consider the second type of $A_6$.
In this case, the centralizer of the $A_6$ in the Monster
is $2\udot L_3(4){:}2$, and the $A_6$ contains elements of Monster class $2B$.
Hence the centralizer of
$3^2{:}2$ is $2\udot U_4(3){:}2^2$, that is an involution centralizer in
$O_8^+(3)$, and the argument of
\cite[Theorem 11.2]{moreB} then shows that any group we obtain in this way lies in
$S_3\times 2\udot Fi_{22}$. Hence, if it has shape $2^2\udot L_3(4)$ then
it centralizes an element of order $3$. In other words, any $L_3(4)$ of this type
in $G$ lies in $S_3\times U_6(2)$.

To put more detail into the argument, the intersection of two copies of $L_3(4)$ in
$U_4(3)$ is either $S_6$ or $2^4S_4$ (in the case when the two copies are conjugate),
or $L_2(7){:}2$ or $2^4{:}A_5$ (when they are not). Lifting to the double covers
we may lose a $2$ from the top of the group. Now the centralizer of 
 $2^2\udot L_3(4)$ in the Monster cannot contain elements of order $5$, so this eliminates
two of the cases. The $L_2(7){:}2$ case gives the well-known group
$L_3(2)\times 2^2\udot L_3(4)$ which we have already seen.
The final case may or may not be $L_3(4)$, but whatever it is has normalizer
contained in a $2$-local subgroup of $G$.

Finally we consider the case of the third type of $A_6$.
This case was omitted in the proof of Theorem 11.2 in \cite{moreB},
perhaps because it was considered obvious,
but more likely due to oversight. The centralizer of this $A_6$ in the Monster is
$M_{11}$. Then from Norton's Monstralizer list \cite{Anatomy1} we read off that the
centralizer of the relevant $3^2{:}2$ is $3^5{:}M_{11}$. Now the intersection 
of two copies of $M_{11}$ in this $3^5{:}M_{11}$ is either $3^2{:}Q_8$ or $A_5$.
In the former case, the Monstralizer of $3^2{:}Q_8$ is again $3^5{:}M_{11}$,
which does not involve $L_3(4)$. In the latter case, the Monstralizer of
$A_5$ is either $A_{12}$ or $M_{11}$ (neither of which involves $L_3(4)$),
or $2.M_{22}.2$, in which the subgroup $2\udot L_3(4)$ does not centralize
a $2A$-pure $2^2$, so does not lie in $2^2\udot{}^2E_6(2)$.

\subsection{Conclusion}
In this section we have shown that there is a unique class of $M_{22}$
in $G.S_3$, and that the only case in which the normalizer of a group
$L_3(4)$, $U_4(3)$ or $U_3(3)$ is maximal is the case of the $L_3(4)$
with normalizer $(L_3(2)\times L_3(4)){:}2$ in $G$.
This concludes the proof of our main results.

Of the $39$ simple groups we had to classify, 
in the following $33$ cases
a complete list up to conjugacy has been obtained:
\begin{itemize}
\item $A_{12}$, $A_{11}$,  $A_{10}$, $A_9$, $A_8$, $A_7$, $A_5$,
\item $L_2(8)$, $L_2(11)$, $L_2(16)$, $L_2(17)$, $L_2(25)$, $L_2(13)$,   
\item $L_3(3)$, $L_4(3)$, 
$U_3(8)$, $U_5(2)$, $U_6(2)$, 
\item
$O_7(3)$, $O_8^+(2)$, $O_8^-(2)$, $S_4(4)$,
$O_{10}^-(2)$,  $S_8(2)$, $S_6(2)$,
\item   $G_2(3)$, $F_4(2)$, ${}^3D_4(2)$, ${}^2F_4(2)'$,
\item $M_{11}$, $M_{12}$, $M_{22}$, $Fi_{22}$.
\end{itemize}
The remaining $6$ have been dealt with to the extent that their normalizers are 
shown to be non-maximal, although a complete list of conjugacy classes
and normalizers has not (yet) been obtained:
\begin{itemize}
\item $A_6$, $L_2(7)$, $U_3(3)$, $U_4(2)$, $U_4(3)$, $L_3(4)$.
\end{itemize}

\section{Further remarks}
\label{further}
In this section we provide an alternative proof of
the following theorem, using neither computation nor Norton's results:
\begin{theorem}
\begin{enumerate}
\item There is a single conjugacy class of each of the groups $O_{10}^-(2)$,
$A_{12}$ and $A_{11}$ in $G$, and the normalizers in $G.S_3$ are
$$(3\times A_{11}){:}2 < (3\times A_{12}){:}2 < (3\times O_{10}^-(2)){:}2.$$
\item There is a single class of $F_4(2)$ in $G.S_3$, splitting into three classes in $G$.
The normalizer in $G.S_3$ is $2\times F_4(2)$.
\item There are two classes of each of the groups
$S_8(2)$ and ${}^3D_4(2)$ in $G.S_3$, splitting into four each classes in $G$.
The normalizers in $G.S_3$ are
\begin{itemize}
\item $S_8(2) \times S_3 < (O_{10}^-(2)\times 3){:}2$;
\item $S_8(2) \times 2 < F_4(2)\times 2$;
\item ${}^3D_4(2){:}3\times S_3$;
\item ${}^3D_4(2){:}3 \times 2< F_4(2)\times 2$.
\end{itemize}
\end{enumerate}
\end{theorem}
\subsection{$F_4(2)$}
Any subgroup isomorphic to $F_4(2)$ may be constructed by taking a group
$L_3(2)\times L_3(2)$ and extending the Sylow $7$-normalizer to $7^2{:}(3\times 2A_4)$.
Now $L_3(4)$ contains exactly three conjugacy classes of $L_3(2)$, so
there are just three possibilities for the $L_3(2)\times L_3(2)$. The extension
is to the full Sylow $7$-normalizer in $G$, so is unique.
Hence there are exactly $3$ conjugacy classes of $F_4(2)$ in $G$, fused in $G.3$.
The full normalizer in $G.S_3$ is $F_4(2)\times 2$.

\subsection{$O_{10}^-(2)$}
Any subgroup $O_{10}^-(2)$ can be constructed from two copies of $A_5\times A_8$
intersecting in $A_5\times A_5\times 3$. Since the $5$-centralizer is $5\times A_8$, it is
obvious that the choices, of the first subgroup $A_5\times A_8$,
and the subgroup $A_5\times A_5\times 3$, and finally the
second $A_8\times A_5$ are unique up to relevant conjugacy at each stage.
Hence there is exactly one conjugacy class of $O_{10}^-(2)$ in $G$, whose full
normalizer in $G.S_3$ is $(O_{10}^-(2)\times 3){:}2$.

\subsection{$A_{12}$}
Any subgroup isomorphic to $A_{12}$ may be constructed from two groups
$(A_5\times A_7){:}2$ intersecting in $S_5\times S_5$. Since the $5$-centralizer
$5\times A_8$ lies in the group $(A_5\times A_8){:}2$, there is a unique
class of $(A_5\times A_7){:}2$, and a unique class of $S_5\times S_5$ in it.
Now in the normalizer of the second $A_5$, we need to extend $S_5\times 2$
to $S_7$ inside $S_8$, and there is obviously a unique way to do this.
Hence there is exactly one conjugacy class of $A_{12}$ in $G$, whose full
normalizer in $G.S_3$ is $(A_{12}\times 3){:}2$. This group is never maximal
in any extension of $G$ by outer automorphisms.

\subsection{$A_{11}$}
A similar argument applies with $(A_5\times A_6){:}2$. We restrict to a subgroup
$(A_5\times A_5){:}2$, in which the two factors are conjugate in $G$.
At the final stage, we have to extend $S_5$ to $S_6$ in $A_8$, and again it is
clear that there is only one way to do this.
Hence there is exactly one conjugacy class of $A_{11}$ in $G$, whose full
normalizer in $G.3$ is $(3\times A_{11}){:}2$. This group is never maximal
in any extension of $G$.

\subsection{${}^3D_4(2)$}
Every such group can be made from $7\times L_3(2)$ by extending the 
Sylow $7$-normalizer to $7^2{:}2A_4$. As shown above, there are exactly
four classes of $7\times L_3(2)$ in $G$, one containing a central $7B$ and three
containing a central $7A$. In each case the extension from $7^2{:}3$ to
$7^2{:}2A_4$ is unique within the full Sylow $7$-normalizer $7^2{:}(3\times 2A_4)$.
Hence there are exactly four conjugacy classes of ${}^3D_4(2)$ in $G$, each with normalizer
${}^3D_4(2){:}3$ in $G$. In $G.3$, three classes are fused and the other is centralized.
Thus the normalizers in $G.S_3$ are respectively ${}^3D_4(2){:}3\times 2$ and
${}^3D_4(2)\times S_3$. The latter is maximal in $G.S_3$ (and its intersection with
$G.3$ is maximal therein).

\subsection{$S_8(2)$}
Any group $S_8(2)$ can be constructed by taking a group $S_3\times S_6(2)$,
restricting to $S_3\times S_3\times S_6$, and then extending to $S_6\times S_6$.
Now there are exactly four classes of $S_3\times S_6(2)$ in $G$, in one of which
the $S_3$ contains $3B$-elements, while in the other three the $S_3$ contains
$3A$-elements. Then the restriction from $S_6(2)$ to $S_3\times S_6$ is unique
up to conjugacy, and the centralizer in $G$ of the $S_6$ is exactly $S_6$.
Hence there is at most one copy of $S_8(2)$ containing any given $S_3\times S_6(2)$.
But we already know there are at least four conjugacy classes of $S_8(2)$ in $G$,
one centralized by an outer $S_3$, and three more in $F_4(2)$ centralized by an
outer involution,
so there are exactly four. Three are fused in $G.3$, while the other is centralized by
an element of class $3D$. The normalizers in $G.S_3$ are respectively
$S_8(2)\times 2$, contained in $F_4(2)\times 2$, and $S_8(2)\times S_3$,
contained in $(O_{10}^-(2)\times 3){:}2$.

\section*{Appendix: $(2C,3C,11)$ generators for ${}^2E_6(2)$}
$$\begin{array}{rrrrrrr}
11&6&24&20&19&20&22\cr
11&8&11&17&20&12&12/19\cr
11&8&11&22&35&12&15\cr
11&8&18&28&17&17&11/30\cr
11&8&22&16&13&22&19/28\cr
11&8&24&20&19&12&18/35\cr
11&8&28&17&14&17&19/24\cr
11&8&35&8&17&14&19/35\cr
11&8&35&24&33&28&10/20\cr
11&9&13&17&17&20&19/28\cr
11&9&19&11&33&17&9/22\cr
11&9&19&35&13&19&20/22\cr
11&9&21&17&17&19&20/20\cr
11&9&22&35&33&21&12/33\cr
11&9&24&12&13&16&35\cr
11&9&24&35&12&17&11\cr
11&9&33&13&19&30&17/28\cr
11&10&19&19&19&33&11/19\cr
11&10&22&18&10&10&24\cr
11&10&22&20&28&16&13/35\cr
11&10&28&17&13&19&13/16\cr
11&10&35&11&28&17&13/20\cr
11&12&11&24&22&24&17/24\cr
11&12&13&24&18&21&22/33\cr
11&12&17&17&14&19&15/16\cr
11&12&18&10&30&28&17/24\cr
11&12&18&19&13&19&22/30\cr
11&12&21&12&10&16&19\cr
11&12&22&21&22&35&13/19\cr
11&12&30&20&11&9&33\cr
11&12&33&13&19&11&16/35\cr
11&12&33&17&22&17&22/35\cr
11&12&35&24&11&17&17\cr
11&13&12&21&33&22&19/35\cr
11&13&16&28&28&17&16/19\cr
11&13&16&30&35&16&21/35\cr
11&13&17&30&17&35&13/20\cr
11&13&17&30&33&18&16/19\cr
11&13&19&17&17&19&19/28\cr
11&13&20&21&20&17&17/24\cr
11&13&21&21&30&18&13/17\cr
11&13&21&24&22&12&19/22\cr
11&13&22&22&22&22&20/22\cr
11&13&24&12&17&13&18/22\cr
11&13&28&21&22&19&17/21\cr
11&13&30&17&22&16&18/33\cr
11&13&30&22&19&24&17/35
\end{array}\qquad
\begin{array}{rrrrrrr}
11&13&33&17&17&17&12/21\cr
11&13&33&22&13&11&19/22\cr
11&14&17&19&28&19&21/33\cr
11&14&17&28&22&17&12/22\cr
11&14&19&20&24&17&17/19\cr
11&14&35&17&19&14&21/22\cr
11&15&9&17&20&19&9/21\cr
11&15&16&30&18&16&13/17\cr
11&15&18&35&13&18&13/13\cr
11&15&19&19&11&19&17/22\cr
11&15&19&30&19&18&17/17\cr
11&15&19&33&11&35&19/22\cr
11&15&20&18&8&16&28\cr
11&15&22&18&33&12&13\cr
11&15&24&22&35&11&20/22\cr
11&15&28&30&33&16&18/22\cr
11&15&30&22&30&22&13/16\cr
11&16&13&17&24&17&12/30\cr
11&16&19&11&17&18&13/20\cr
11&16&22&14&22&15&21/35\cr
11&16&22&17&17&17&21/35\cr
11&16&22&17&22&17&20/21\cr
11&16&30&28&35&19&22/30\cr
11&16&33&11&17&22&22/33\cr
11&17&12&17&22&19&17/21\cr
11&17&12&17&24&18&12/17\cr
11&17&13&33&21&33&14/22\cr
11&17&16&6&13&16&13\cr
11&17&16&12&35&18&12\cr
11&17&16&13&20&30&16/22\cr
11&17&19&35&14&19&17/30\cr
11&17&24&19&28&18&12/35\cr
11&17&28&24&30&19&20\cr
11&17&33&21&16&33&16/21\cr
11&17&33&24&28&12&22/22\cr
11&18&9&14&20&30&33\cr
11&18&11&13&28&18&13/21\cr
11&18&16&20&30&35&16/18\cr
11&18&16&35&13&22&18/30\cr
11&18&16&35&24&17&21/35\cr
11&18&17&8&16&33&17\cr
11&18&17&17&18&22&18/19\cr
11&18&17&21&16&12&19/21\cr
11&18&17&33&8&18&22/28\cr
11&18&17&35&18&28&11/17\cr
11&18&18&13&17&21&18/35\cr
11&18&19&11&30&11&21/30
\end{array}$$
$$\begin{array}{rrrrrrr}
11&18&19&15&11&19&13/30\cr
11&18&19&18&19&17&19/30\cr
11&18&20&19&17&22&20/22\cr
11&18&20&28&15&17&17/18\cr
11&18&21&16&33&12&13/24\cr
11&18&21&28&16&14&21/22\cr
11&18&22&19&19&20&17/33\cr
11&18&24&22&21&13&17/30\cr
11&18&24&30&19&28&18\cr
11&18&28&18&12&24&33/35\cr
11&18&28&28&33&22&16/18\cr
11&18&30&12&17&17&12/17\cr
11&18&30&19&17&19&19/28\cr
11&18&30&35&17&19&12/33\cr
11&18&33&18&20&24&17/18\cr
11&18&35&11&19&18&14/17\cr
11&18&35&22&17&22&30/33\cr
11&20&10&11&21&19&16/19\cr
11&20&11&4&14&33&12/17\cr
11&20&11&16&17&13&13/13\cr
11&20&12&13&33&16&9/19\cr
11&20&13&17&18&20&10/19\cr
11&20&13&22&22&11&14/17\cr
11&20&13&35&33&21&17/18\cr
11&20&16&13&17&13&18/24\cr
11&20&16&17&30&18&17/22\cr
11&20&17&19&13&22&17/21\cr
11&20&17&19&22&17&8/17\cr
11&20&17&22&15&18&13/19\cr
11&20&17&24&28&33&22/24\cr
11&20&17&30&24&28&8/24\cr
11&20&17&35&19&22&35/35\cr
11&20&19&11&16&9&12/33\cr
11&20&19&33&30&24&19/24\cr
11&20&22&17&13&22&22/33\cr
11&20&22&19&15&17&21\cr
11&20&22&19&33&15&11/18\cr
11&20&22&30&10&30&20/33\cr
11&20&28&19&14&20&13/17\cr
11&20&30&24&13&19&9/18\cr
11&20&33&16&30&22&22/28\cr
11&21&9&16&16&17&35\cr
11&21&12&17&12&13&18/28\cr
11&21&13&22&19&30&11/17\cr
11&21&16&16&10&33&22/24\cr
11&21&16&20&24&17&20/30\cr
11&21&18&17&30&20&14/35\cr
11&21&18&33&17&28&20/30\cr
11&21&18&35&18&13&18/19
\end{array}\qquad
\begin{array}{rrrrrrr}
11&21&19&11&16&19&17/30\cr
11&21&19&21&30&33&21/21\cr
11&21&19&33&16&35&16/19\cr
11&21&20&20&24&13&19/35\cr
11&21&20&33&11&22&17/19\cr
11&21&21&12&12&15&19/30\cr
11&21&21&17&19&13&16/24\cr
11&21&21&28&12&22&21/28\cr
11&21&22&15&13&33&12\cr
11&21&22&17&13&11&15/17\cr
11&21&22&17&24&20&17/28\cr
11&21&24&17&21&24&20/33\cr
11&21&30&13&16&12&19/35\cr
11&21&33&19&28&17&17/18\cr
11&21&33&22&30&30&13/30\cr
11&21&35&13&30&19&17/22\cr
11&21&35&35&11&16&30/35\cr
11&21&35&35&30&30&13/33\cr
11&24&12&28&35&18&16/20\cr
11&24&13&16&24&19&16/19\cr
11&24&13&33&18&20&19/33\cr
11&24&16&16&18&22&24/28\cr
11&24&17&13&13&19&19/21\cr
11&24&17&14&18&13&22/28\cr
11&24&17&22&13&24&17/28\cr
11&24&18&30&13&28&16/17\cr
11&24&19&11&18&17&19/19\cr
11&24&19&16&16&20&14/18\cr
11&24&19&19&16&18&19/33\cr
11&24&19&22&28&24&16/21\cr
11&24&19&35&13&30&11/17\cr
11&24&20&13&18&13&22/33\cr
11&24&20&18&35&33&16/28\cr
11&24&21&9&24&22&11/33\cr
11&24&21&12&12&19&20/33\cr
11&24&21&17&30&14&13/19\cr
11&24&21&18&33&35&18/24\cr
11&24&21&19&21&24&19/20\cr
11&24&21&21&18&35&22/33\cr
11&24&22&13&33&30&15/30\cr
11&24&22&19&21&28&13/17\cr
11&24&24&16&17&16&14/21\cr
11&24&24&24&17&12&18/28\cr
11&24&28&13&19&19&17/18\cr
11&24&28&17&35&19&18/20\cr
11&24&28&28&28&13&20/22\cr
11&24&30&16&18&13&11/24\cr
11&24&30&16&22&17&19/20\cr
11&24&33&11&28&16&19/21
\end{array}$$
$$\begin{array}{rrrrrrr}
11&24&35&8&16&35&17/33\cr
11&24&35&11&18&13&19/35\cr
11&24&35&11&35&30&19/21\cr
11&24&35&13&18&21&21/28\cr
11&24&35&22&13&24&12/17\cr
11&28&13&19&16&21&16/21\cr
11&28&13&19&24&18&19/20\cr
11&28&16&12&30&17&30\cr
11&28&16&19&21&19&33/33\cr
11&28&17&14&17&18&16/24\cr
11&28&17&22&13&28&21/22\cr
11&28&17&22&20&14&22/28\cr
11&28&18&19&17&12&16/22\cr
11&28&18&20&17&16&16/21\cr
11&28&18&21&12&33&19/22\cr
11&28&19&35&35&30&11/17\cr
11&28&20&20&35&33&16/33\cr
11&28&20&28&24&17&19/21\cr
11&28&20&28&35&15&15/18\cr
11&28&21&19&18&18&19/19\cr
11&28&21&28&24&35&16/24\cr
11&28&22&13&17&28&13/17\cr
11&28&22&19&28&14&12/18
\end{array}\qquad
\begin{array}{rrrrrrr}
11&28&22&19&33&22&17/22\cr
11&28&22&21&30&14&18/30\cr
11&28&24&12&13&16&8/16\cr
11&28&24&13&35&35&21/35\cr
11&28&24&19&22&17&33/33\cr
11&28&24&30&22&28&11/13\cr
11&28&30&28&12&24&12/18\cr
11&28&33&24&22&13&13/17\cr
11&28&35&22&21&17&17/19\cr
11&30&9&19&24&22&17/24\cr
11&30&9&21&30&28&9/18\cr
11&30&13&22&22&19&17/28\cr
11&30&16&13&16&19&16/21\cr
11&30&17&14&22&21&21/21\cr
11&30&17&30&11&22&13/24\cr
11&30&18&13&9&28&19/22\cr
11&30&18&35&35&22&18/20\cr
11&30&22&13&18&14&22/24\cr
11&30&30&8&13&24&19/21\cr
11&30&30&19&14&20&16/17\cr
11&30&35&8&16&16&28\cr
11&30&35&11&18&17&17/35\cr
11&30&35&28&18&21&19/22
\end{array}$$


\begin{thebibliography}{99}
\bibitem{BHRD}
J. N. Bray, D. F. Holt and C. Roney-Dougal,
{\it The maximal subgroups of the low-dimensional finite 
classical groups},
LMS Lecture Notes Ser. 407, Cambridge UP, 2013.


\bibitem{Atlas}
J. H. Conway, R. T. Curtis, S. P. Norton, R. A. Parker and R. A. Wilson,
{\it An Atlas of Finite Groups}, Oxford University Press, 1985.

\bibitem{GAP}
The GAP group, GAP -- Groups, Algorithms, and Programming, Version 4.8.10; 2018.
(https://www.gap-system.org)

\bibitem{ABC} C. Jansen, K. Lux, R. A. Parker and R. A. Wilson,
{\it An Atlas of Brauer Characters},
Oxford UP, 1995.

\bibitem{KWF22} P. B. Kleidman and R. A. Wilson,
The maximal subgroups of $Fi_{22}$,
{\it Math. Proc.Cambridge Philos. Soc.} {\bf 102} (1987), 17--23.
Corrigendum, {\it ibid.} {\bf 103} (1988), 383.

\bibitem{Anatomy1} S. Norton, { Anatomy of the Monster: I,} in
{\it Proceedings of the
{\sc Atlas} Ten Years On conference (Birmingham 1995)}, pp.\ 198--214, Cambridge
Univ.\ Press, 1998.

\bibitem{F42} S. P. Norton and R. A. Wilson,
The maximal subgroups of $F_4(2)$ and its automorphism group,
{\it Comm. Algebra} {\bf 17} (1989), 2809--2824.


\bibitem{Anatomy2} S. P. Norton and R. A. Wilson, Anatomy of the Monster: II,
{\it Proc. London Math. Soc.} {\bf 84} (2002), 581--598.

\bibitem{WilF22} R. A. Wilson,
On maximal subgroups of the Fischer groups $Fi_{22}$,
{\it Math. Proc. Cambridge Philos. Soc.} {\bf 95} (1984), 197--222.

\bibitem{MaxAut} R. A. Wilson, Maximal subgroups of automorphism groups
of simple groups, {\it J. London Math. Soc.} {\bf 32} (1985), 460--466.


\bibitem{Filoc} R. A. Wilson,
The local subgroups of the Fischer groups,
{\it J. London Math. Soc.} {\bf 36} (1987), 77--94.

\bibitem{someB} R. A. Wilson,
Some subgroups of the Baby Monster,
{\it Invent. Math.} {\bf 89} (1987), 197--218.

\bibitem{moreB} R. A. Wilson,
More on maximal subgroups of the Baby Monster,
{\it Arch. Math. (Basel)} {\bf 61} (1993), 497--507.

\bibitem{genusB} R. A. Wilson, The symmetric genus of the Baby Monster,
{\it Quart. J. Math. (Oxford)} {\bf 44} (1993), 513--516.

\bibitem{maxB} R. A. Wilson,
The maximal subgroups of the Baby Monster, I,
{\it J. Algebra} {\bf 211} (1999), 1--14.

\bibitem{TFSG} R. A. Wilson,
{\it The finite simple groups}, Springer GTM 251, 2009.



\bibitem{webatlas} R. A. Wilson et al., {\it An Atlas of Group Representations},
http://brauer.maths.qmul.ac.uk/Atlas/.


\end{thebibliography}
\end{document}